\documentclass[12pt]{amsart}

\usepackage{amsthm}
\newtheorem{theorem}{Theorem}[section]
\newtheorem{lemma}[theorem]{Lemma}
\newtheorem{proposition}[theorem]{Proposition}
\theoremstyle{definition}

\theoremstyle{remark}

\counterwithin*{section}{part}

\usepackage{amssymb}
\usepackage{empheq}
\usepackage{stmaryrd}
\usepackage{enumerate}
\usepackage[shortlabels]{enumitem}
\usepackage{calc}
\usepackage{url}
\usepackage{comment}

\newcommand{\CAT}{{\rm{CAT}(0)}}

\def\EE{\mathbb{E}}
\def\PP{\mathbb{P}}

\usepackage[left=2.0cm,%
                right=2.0cm,%
                top=2.5cm,%
                bottom=3.5cm,%
                headheight=12pt,%
                a4paper]{geometry}%

\begin{document}

\title[On stochastic Busemann subgradient methods]{On Busemann subgradient methods for stochastic minimization in Hadamard spaces}

\author[N. Pischke]{Nicholas Pischke}
\date{\today}
\maketitle
\vspace*{-5mm}
\begin{center}
{\scriptsize 
Department of Computer Science, University of Bath,\\
Claverton Down, Bath, BA2 7AY, United Kingdom.\\
E-mail: nnp39@bath.ac.uk}
\end{center}

\maketitle
\begin{abstract}
We study the recently introduced Busemann subgradient method due to Goodwin, Lewis, L\'opez-Acedo and Nicolae, extending it to minimize the mean of a stochastic function over general Hadamard spaces. We prove a strong convergence theorem under a local compactness assumption and further prove weak convergence of the method over separable Hadamard spaces. The proof is based on a general (weak) convergence theorem for stochastic processes in Hadamard spaces which confine to a stochastic variant of quasi-Fej\'er monotonicity, together with a nonlinear variant of Pettis' theorem, which are of independent interest. Lastly, we provide best function value estimates as well as a strong convergence result under a strong convexity assumption, and in that latter case in particular derive explicit rates of convergence. We illustrate the practical use of the general mean minimization problem and the associated Busemann subgradient method through statistical models for phylogenetic inference via the Billera-Holmes-Vogtmann tree space.
\end{abstract}

\noindent
{\bf Keywords:} Subgradient method, Busemann function, stochastic approximation, weak convergence, Hadamard spaces, proof mining\\ 
{\bf MSC2020 Classification:} 47J25, 90C15, 90C25, 62L20, 03F10

\section{Introduction}

\subsection{Background and motivation}

One of the most general and productive formulations of (stochastic) approximation is the problem of minimizing a convex integral function, that is solving the problem
\[
\min_{x\in C}\int f(e,x)\,d\mu(e),\tag{P}\label{problem}
\]
for a given normal convex integrand $f:E\times C\to \mathbb{R}$ (see \cite{RockafellarWets1998}) on a complete probability space $(E,\mathcal{E},\mu)$ and some target set $C\subseteq X$ in a suitable space $X$, say a Hilbert or Banach space. There are various prevalent modern tools for approaching this problem, among them being stochastic variants of the proximal point algorithm and of projected subgradient methods, and we refer to \cite{Bertsekas2011,Bertsekas2012,NemirovskiJuditskyLanShapiro2009}, among many others, for various such discussions.

Indeed, most of these methods already are concerned with the particularly important and motivating special case of minimizing a finite sum of convex functions $f_i:C\to\mathbb{R}$, i.e.
\[
\min_{x\in C}\sum_{i=1}^m f_i(x).\tag*{(P)$_f$}\label{problemFin}
\]
These problems gain further relevance if considered outside of linear contexts such as Hilbert or Banach spaces. Concretely, latest since the extensive developments of machine learning in recent years, where optimization over nonlinear spaces such as manifolds plays a key role (we refer e.g.\ to the discussion in \cite{ZhangSra2016}), extensions of these tools from (stochastic) convex analysis to nonlinear contexts are of high practical relevance. Further, this relevance naturally transcends the realm of spaces with differentiable structure such as manifolds, as illustrated by e.g.\ the Billera-Holmes-Vogtmann tree space \cite{BilleraHolmesVogtmann2001} prominently used in phylogenetics. 

The present paper is concerned with a method recently introduced by Goodwin, Lewis, L\'opez-Acedo and Nicolae \cite{GoodwinLewisLopezAcedoNicolae2025} to solve the above minimization problem \ref{problemFin} in the context of the general class of geodesic metric spaces with nonpositive curvature, as introduced in the work of Alexandrov \cite{Aleksandrov1951}. These spaces, often called $\CAT$ spaces after the work of Gromov \cite{Gromov1987} and Hadamard spaces if complete, uniformly cover examples such as Hilbert spaces, $\mathbb{R}$-trees and Hadamard manifolds (i.e.\ complete simply connected Riemannian manifolds of nonpositive sectional curvature) and the Billera-Holmes-Vogtmann tree space mentioned before, as well as further involved examples. As such, they have continuously been a focus of attention for extending tools from convex analysis on linear spaces to nonlinear contexts, most notably recently in other related work by Goodwin, Lewis, L\'opez-Acedo and Nicolae  \cite{GoodwinLewisLopezAcedoNicolae2025b,LewisLopezAcedoNicolae2024,LewisLopezAcedoNicolae2024b}. We refer to \cite{AlexanderKapovitchPetrunin2023,BridsonHaefliger1999} for a comprehensive overview of geodesic and $\CAT$ spaces and further refer to \cite{Bacak2014a} for a shorter treatment focused on aspects of convex analysis and optimization.

Concretely, the work \cite{GoodwinLewisLopezAcedoNicolae2025} provides a projected subgradient method for the problem \ref{problemFin}. However, lack of linear structure makes constructions such as subgradients, which naturally rely on duality theory, complicated in general geodesic metric spaces such as Hadamard spaces. In that way, the approach of \cite{GoodwinLewisLopezAcedoNicolae2025} diverges from usual subgradient methods by making particularly novel use of the boundary cone $CX^\infty$ of a Hadamard space $X$ and associated Busemann functions, and other advanced geometric tools from Hadamard spaces, to define a novel type of subgradient for the associated functions.

Indeed, as shown in \cite{GoodwinLewisLopezAcedoNicolae2025}, this resulting notion of a Busemann subgradient, which in Euclidean space coincides with the usual notion of subgradients, supports a broad theory and in particular allows for the derivation and effective analysis of the following stochastic projected subgradient method:\footnote{The work \cite{GoodwinLewisLopezAcedoNicolae2025} also considers a non-stochastic incremental variant of this Busemann subgradient method, but we will not be concerned with such deterministic methods in the present paper.} Given a starting point $x_0\in C$, an i.i.d.\ sequence $(i_n)$ of random selections distributed uniformly over $\{1,\dots,m\}$ and a sequence of step-sizes $(t_n)$ of positive reals satisfying suitable growth conditions, one defines the iteration
\[
x_{n+1}:=P_C(r_{x_n,\xi_{n}}(s_nt_n))\tag{\textsf{SB}$_0$}\label{SB0}
\]
where $[\xi_n,s_n]=\mathsf{Busemann}_{f_{i_n}}(x_{n})\in CX^\infty$ represents a Busemann subgradient of $f_{i_n}$ at $x_n$, chosen using an oracle function $\mathsf{Busemann}$, and $r_{x_n,\xi_{n}}(s_nt_n)$ represents the point reached by following the geodesic ray starting from $x_n$ with direction $\xi_n$ and speed $s_n$, and with step size $t_n$.\footnote{The above is just intended as a sketch. In particular, all these currently undefined notions will be detailed in Section \ref{initial} later on.} 

For this stochastic method, they in particular establish the following result:

\begin{theorem}[Theorem 6.2 in \cite{GoodwinLewisLopezAcedoNicolae2025}]\label{old}
Let $X$ be a Hadamard space with the geodesic extension property and at least two points and let $C\subseteq X$ be nonempty, closed and convex. Let $F(x):=\sum_{i=1}^m f_i(x)$ be given, where each $f_i:C\to\mathbb{R}$ is Busemann subdifferentiable, and such that $\mathrm{argmin}F\neq\emptyset$. Assume further that there is some $L> 0$ such that every Busemann subgradient $[\xi,s]$ of any $f_i$ at any point in $C$ satisfies $s\leq L$. Let $(x_n)$ be the sequence generated by \eqref{SB0} with $(t_n)$ and $(i_n)$, where $(t_n)\subseteq (0,+\infty)$ is such that $\sum_{k=0}^nt_k^2/\sum_{k=0}^nt_k\to 0$, and $(i_n)$ is a sequence of i.i.d.\ random variables distributed uniformly over $\{1,\dots,m\}$.

Then $\EE\left[\min_{i=0,\dots,n}F(x_i)\right]\to \min F$. Moreover, if $C$ has diameter bounded by $D>0$, one has
\[
\EE\left[\min_{i=0,\dots,n}F(x_i) -\min F\right]\leq \frac{2(1+\log(3))mLD}{\sqrt{n+2}}
\]
for all $n\geq 2$, using the parameter sequence $t_n:=\frac{D}{mL\sqrt{n+1}}$.
\end{theorem}

While the above result presents a (quantitative) approximation result for the method \eqref{SB0}, the possibilities of Busemann subgradients and methods such as \eqref{SB0} are only beginning to be explored (see e.g.\ also the recent \cite{CriscitielloKim2025,Hirai2026}), and many questions remain, such as convergence behavior of the whole sequence, rates under regularity conditions, and extensions for solving general stochastic minimization problems. We address precisely those questions in the present paper.

\subsection{Main results and related work}

In this paper, we adapt the previous Busemann subgradient method introduced by Goodwin, Lewis, L\'opez-Acedo and Nicolae in \cite{GoodwinLewisLopezAcedoNicolae2025} to solve general stochastic minimization problems as in \eqref{problem}, and moreover provide a more detailed study on its asymptotic behavior. 

Concretely, let $(\Omega,\mathcal{F},\PP)$ and $(E,\mathcal{E},\mu)$ be probability spaces, with $(E,\mathcal{E},\mu)$ complete, and let $X$ be a separable Hadamard space with the geodesic extension property and at least two points.\footnote{These assumptions pertain to the existence of boundary points and the well-definedness of various notions, as discussed in detail later.} Further, fix a closed, convex and nonempty subset $C\subseteq X$ together with a functional $f:E\times C\to \mathbb{R}$.

We now want to solve \eqref{problem}, that is we want to minimize $F(x):=\int f(e,x)\,d\mu(e)$ over $C$, assuming that $F$ is well-defined and that such a minimum actually exists, i.e.\ that $\mathrm{argmin}F\neq\emptyset$.

For that, we assume that $f$ has the following properties: At first, we assume that $f$ actually has Busemann subgradients and further is suitably measurable:
\[
\begin{cases}f(e,\cdot)\text{ is Busemann subdifferentiable for any }e\in E\\
\text{and }f(\cdot,x)\text{ is measurable for all }x\in C.\end{cases}\tag{\textsf{A1}}\label{A1}
\]
Further, we impose a Lipschitz condition on the Busemann subgradients:
\[
\begin{cases}
\text{There exists an } L\in L^2(E,(0,\infty),\mu)\text{ with }\int L\,d\mu>0\text{ such that for any }e\in E\text{ and}\\
\text{any Busemann subgradient }[\xi,s]\in CX^\infty\text{ of }f(e,\cdot)\text{ at }x\in C\text{, we have }s\leq L(e).
\end{cases}\tag{\textsf{A2}}\label{A2}
\]
These two assumptions on $f$ are derived from \cite{GoodwinLewisLopezAcedoNicolae2025} (see Assumption A therein, and recall Theorem \ref{old}), suitably extended to the present, more general context. However, both are also natural if seen in the broader context of (stochastic) subgradient methods, where measurability and a Lipschitz condition on the subgradients are among the most basic common assumptions (see e.g.\ \cite{NemirovskiJuditskyLanShapiro2009}). As mentioned before, further details on all involved objects will be given in Section \ref{initial} later on.

Motivated by \cite{GoodwinLewisLopezAcedoNicolae2025}, we consider the following stochastic Busemann subgradient method: Define
\[
x_{n+1}:=P_C(r_{x_n,\xi_{n}}(s_nt_n))\tag{\textsf{SB}}\label{SB}
\]
given a starting point $x_0\in C$ and sequences $(t_n)$ of positive reals as well as $(\zeta_{n+1})$ of random variables $\Omega\to E$, where $[\xi_n,s_n]=\mathsf{Busemann}_f(\zeta_{n+1},x_{n})$ is chosen using an oracle function $\mathsf{Busemann}_f$ for $f$. Concretely, the oracle $\mathsf{Busemann}_f$ is supposed to associate a function $\mathsf{Busemann}_f(\zeta,x):\Omega\to CX^\infty$ to input functions $\zeta:\Omega\to E$ and $x:\Omega\to C$ in such a way that it almost surely selects Busemann subgradients and preserves measurability along the iteration:
\[\begin{cases}\text{For any $n\in\mathbb{N}$}, \mathsf{Busemann}_f(\zeta_{n+1},x_n):\Omega\to CX^\infty\text{ is $\sigma(\zeta_{n+1},x_n)$-measurable}\\
\text{and almost surely represents a Busemann subgradient of $f(\zeta_{n+1},\cdot)$ at $x_n$}.\end{cases}\tag{\textsf{A3}}\label{A3}
\]
While similar such oracles also already appear in \cite{GoodwinLewisLopezAcedoNicolae2025} as discussed before, measurability thereof or of the associated iteration is not discussed, which however will occupy us here. Complementing and extending \cite{GoodwinLewisLopezAcedoNicolae2025}, we will show that an oracle satisfying \eqref{A3} always exists for separable Hadamard spaces, by means of an even stronger result regarding Busemann subdifferentials (see Proposition \ref{oracleExistence} later on) which can be further improved over locally compact spaces (see Proposition \ref{oracleExistenceProper} later on) or when $(E,\mathcal{E},\mu)$ is discrete (see Proposition \ref{oracleDiscrete} later on).
 
Lastly, for the parameters we assume that 
\[
(\zeta_{n+1})\text{ is i.i.d.\ with distribution $\mu$ and }\sum_{n\in\mathbb{N}}t_n=+\infty, \sum_{n\in\mathbb{N}}t_n^2<+\infty.\tag{\textsf{Par}}\label{parameters}
\]
We will later show that the assumptions \eqref{A1} -- \eqref{A3} together with \eqref{parameters} suffice to guarantee that $x_n$ is $\sigma(\zeta_{1},\dots,\zeta_n)$-measurable for any $n\in\mathbb{N}$.

It should be noted that, as mentioned before, \eqref{problem} is a generalisation of the minimization problem for finite sums of functions \ref{problemFin} as studied in \cite{GoodwinLewisLopezAcedoNicolae2025}, which is reobtained by considering discrete probability spaces $(E,\mathcal{E},\mu)$. Further, the above method \eqref{SB} subsumes the method \eqref{SB0} in the same way, so that the present results also pertain to the method studied in \cite{GoodwinLewisLopezAcedoNicolae2025}.

Towards an asymptotic analysis of the method \eqref{SB}, we will show the following results. At first, we show that the sequence strongly converges in the presence of a local compactness assumption.

\begin{theorem}\label{SBstrongConv}
Let $(E,\mathcal{E},\mu)$ and $(\Omega,\mathcal{F},\PP)$ be probability spaces, with $(E,\mathcal{E},\mu)$ complete, and let $X$ be a locally compact Hadamard space with the geodesic extension property and at least two points. Let $C\subseteq X$ be closed, convex and nonempty. Let $f:E\times C\to \mathbb{R}$ be a function with properties \eqref{A1} -- \eqref{A3} as above. Define $F(x):=\int f(e,x)\,d\mu(e)$ and assume $\mathrm{argmin}F\neq\emptyset$. Let $(x_n)$ be the iteration given by \eqref{SB} with $(t_n)$ and $(\zeta_{n+1})$, and assume \eqref{parameters}.

Then $(x_n)$ a.s.\ strongly converges to an $\mathrm{argmin}F$-valued random variable.
\end{theorem}

Lifting the local compactness assumption, the best we in general can expect is a weak convergence result. Extending such a result to a stochastic context like the above is rather nontrivial. Indeed, for related methods like the projected stochastic gradient algorithm for convex and $L^2$-Fr\'echet differentiable objective functionals over Hilbert spaces (with convex constraint sets), only the relatively recent work of Geiersbach and Pflug \cite{GeiersbachPflug2019} established weak almost sure convergence. Analogously, also the almost sure weak convergence of the stochastic proximal point algorithm over Hadamard spaces \cite{Bacak2018} (generalising related work on the minimization of a finite sum of functions as discussed in the well-known work \cite{Bacak2014b}, in a similar way as the present paper generalizes \cite{GoodwinLewisLopezAcedoNicolae2025}) remained an open problem (even over separable Hilbert spaces, see \cite{Bacak2023}) until the recent work of the author \cite{Pischke2026}. Utilizing a particular result developed in \cite{Pischke2026} to control the mean function values along the iteration, we here show the following similar result:

\begin{theorem}\label{SBweakErgodicConv}
Let $(E,\mathcal{E},\mu)$ and $(\Omega,\mathcal{F},\PP)$ be probability spaces, with $(E,\mathcal{E},\mu)$ complete, and let $X$ be a separable Hadamard space with the geodesic extension property and at least two points. Let $C\subseteq X$ be closed, convex and nonempty. Let $f:E\times C\to \mathbb{R}$ be a function with properties \eqref{A1} -- \eqref{A3} as above. Define $F(x):=\int f(e,x)\,d\mu(e)$ and assume $\mathrm{argmin}F\neq\emptyset$. Let $(x_n)$ be the iteration given by \eqref{SB} with $(t_n)$ and $(\zeta_{n+1})$, and assume \eqref{parameters}.

Then $(x_n)$ a.s.\ weakly converges to an $\mathrm{argmin}F$-valued random variable.
\end{theorem}

Our arguments for both theorems above are based on a general (weak) convergence theorem for stochastic processes in Hadamard spaces which confine to a stochastic variant of quasi-Fej\'er monotonicity (see Proposition \ref{weakConvergence} later on), which is also of independent interest. Fej\'er monotonicity is a fundamental notion in modern (convex) analysis and optimization, with many if not most iterative procedures confining to a variant thereof (see e.g.\ \cite{Combettes2001,Combettes2009} and also \cite{BauschkeCombettes2017}). In particular, Fej\'er monotonicity has played a crucial role in the weak convergence proof for the deterministic proximal point method in Hadamard spaces given in \cite{Bacak2013}, which relied on a previous (deterministic) weak convergence result for such sequences given in \cite{BacakSearstonSims2012}. Stochastic variants of (quasi-)Fej\'er monotonicity go back at least to the work of Ermol'ev \cite{Ermolev1969,Ermolev1971,ErmolevTuniev1968}, but the notion was subsequently further refined in the seminal work of Combettes and Pesquet \cite{CombettesPesquet2015} (see also \cite{CombettesPesquet2019}), where a general and abstract (weak) convergence theorem for such sequences is established over separable Hilbert spaces on which we have modeled our result on. Indeed, the approach to our central weak convergence result given in Proposition \ref{weakConvergence} is a rather immediate synthesis of the stochastic work \cite{CombettesPesquet2015} set in Hilbert spaces, likewise crucially relying on the seminal Robbins-Siegmund theorem on almost-supermartingale convergence \cite{RobbinsSiegmund1971}, and the deterministic work of Ba\v{c}\'ak, Searston and Sims \cite{BacakSearstonSims2012} set in a metric context. In the course of these results, we further rely on a nonlinear variant of (the consequence of) Pettis' theorem \cite{Pettis1938}, stating that a weak limit of a sequence of measurable functions taking values in a separable Hadamard space is again measurable (see Proposition \ref{nonlinearPettis} later on). To our knowledge, this result is also new to the literature of Hadamard spaces.

The weak almost sure convergence theorem given in Theorem \ref{SBweakErgodicConv} in particular further relies on a stochastic extension of a result due to Alber, Iusem and Solodov \cite{AlberIusemSolodov1998} (see Proposition 2 therein, as well as \cite{NeriPowell2023}) recently established by the author \cite{Pischke2026} (see Lemma 3.4 therein, and Lemma \ref{LipschitzSum} later on) in the context of a similar weak almost sure convergence theorem for the stochastic proximal point method already mentioned above. More concretely, similar to \cite{Pischke2026}, this result is also here used to establish the almost sure convergence of the mean function values of the process towards the minimal value.\footnote{As such, Lemma \ref{LipschitzSum} is related but ultimately slightly different from a similar stochastic variant of the result of Alber, Iusem and Solodov \cite{AlberIusemSolodov1998} derived by Geiersbach and Pflug \cite{GeiersbachPflug2019} (see Proposition 3.2 therein) mentioned previously. In particular, compared to \cite{GeiersbachPflug2019} which relies on an additive representation of stochastic errors, Lemma \ref{LipschitzSum} is particularly tailored to use in a nonlinear setting such as that of the present paper.}

Further considerations on stochastic quasi-Fej\'er monotonicity in a metric context can e.g.\ be found in the recent works \cite{NeriPischkePowell2025,PischkePowell2026}, which while phrased in the context of strong stochastic regularity conditions (which among other things induce strong convergence) nevertheless further illustrate the range of different methods that immediately fall under this paradigm. Indeed, even though we only consider the above method \eqref{SB} in the present paper, we think that the general convergence result formulated here will be of use for the convergence analysis of further methods from stochastic convex optimization over nonlinear spaces (similar to the way as it already found an application in the recent work of the author \cite{Pischke2026} on the stochastic proximal point algorithm in Hadamard spaces).

One such method which we in particular want to highlight derives from the work of Bianchi \cite{Bianchi2016}, in particular its recent extension to a metric setting of nonpositive curvature given in \cite{Pischke2025}. Concretely, the work \cite{Bianchi2016} studies a corresponding stochastic proximal point algorithm for stochastically perturbed monotone operators over separable Hilbert spaces. Only almost sure weak ergodic convergence is known in Hilbert spaces under additional uniform integrability conditions (see \cite{Bianchi2016}), and it follows from results of Brezis and Lions \cite{BrezisLions1978} that this result of Bianchi cannot be extended to (weak) almost sure convergence for general operators in Hilbert spaces. In the context of a strong monotonicity assumption, this convergence result can be improved to almost sure strong convergence of the original iteration. This latter result was extended to a metric setting in \cite{Pischke2025}, and it remains an interesting open problem if also the almost sure weak ergodic convergence proven in \cite{Bianchi2016} can be lifted to the metric setting. To prepare for such an application, we also provide a general and abstract weak ergodic convergence theorem (see Proposition \ref{weakErgodicConvergence} later on), i.e.\ a convergence theorem for the ergodic average $(\overline{x}_n)$ of a sequence $(x_n)$ of points and a sequence of parameters $(t_n)$, defined recursively by 
\[
\overline{x}_0:=x_0\text{ and }\overline{x}_{n+1}:=\frac{\overline{t}_n}{\overline{t}_{n+1}}\overline{x}_n\oplus \frac{t_{n+1}}{\overline{t}_{n+1}}x_{n+1}
\]
where $\overline{t}_n:=\sum_{k=0}^nt_k$, and
writing $(1-\lambda)x\oplus \lambda y$ for the point $\gamma(\lambda d(x,y))$ on the unique geodesic $\gamma:[0,d(x,y)]\to X$ joining $x$ and $y$.

This result however comes at the expense of a narrower class of spaces, being derived for separable Hadamard spaces satisfying condition $(\overline{Q}_4)$ as introduced in \cite{Kakavandi2013}, a slight extension of the $(Q_4)$ condition of Kirk and Panyanak \cite{KirkPanyanak2008}, which in particular includes Hilbert spaces, $\mathbb{R}$-trees and spaces of constant curvature, but is known to exclude certain gluings of Hadamard spaces (see \cite{EspinolaFernandezLeon2009}). Over Hilbert spaces, such a general weak ergodic convergence theorem is due to Passty \cite{Passty1979} and related arguments over Hadamard spaces, involving condition $(\overline{Q}_4)$, have recently been given in the work of Khatibzadeh and Moosavi \cite{KhatibzadehMoosavi2023}, on which our ergodic convergence result is based. In particular, we also show that weak convergence of a sequence implies weak ergodic convergence over Hadamard spaces satisfying condition $(\overline{Q}_4)$ (see Lemma \ref{weakToErgodic} later on), so that Theorem \ref{SBweakErgodicConv} in particular entails the weak ergodic convergence of \eqref{SB} in such spaces.

We are also concerned with quantitative results for the method \eqref{SB}, where we first extend the estimate for the best function values obtained in \cite{GoodwinLewisLopezAcedoNicolae2025} (recall Theorem \ref{old} above):

\begin{theorem}\label{newResidual}
Let $(E,\mathcal{E},\mu)$ and $(\Omega,\mathcal{F},\PP)$ be probability spaces, with $(E,\mathcal{E},\mu)$ complete, and let $X$ be a separable Hadamard space with the geodesic extension property and at least two points. Let $C\subseteq X$ be closed, convex and nonempty. Let $f:E\times C\to \mathbb{R}$ be a function with properties \eqref{A1} -- \eqref{A3} as above. Define $F(x):=\int f(e,x)\,d\mu(e)$ and assume $\mathrm{argmin}F\neq\emptyset$. Let $(x_n)$ be the iteration given by \eqref{SB} with $(t_n)$ and $(\zeta_{n+1})$, assume \eqref{parameters},\footnote{For $(t_n)$, it in fact suffices to assume $\sum_{k=0}^nt_k^2/\sum_{k=0}^nt_k\to 0$, similar to Theorem \ref{old}.} and write $\underline{L}:=\int L^2\,d\mu$.

Then $\EE\left[\min_{i=0,\dots,n}F(x_i)\right]\to \min F$. Moreover, if $C$ has diameter bounded by $D>0$, one has
\[
\EE\left[\min_{i=0,\dots,n}F(x_i) -\min F\right]\leq \frac{2(1+\log(3))D\sqrt{\underline{L}}}{\sqrt{n+2}}
\]
for all $n\geq 2$, using the parameter sequence $t_n:=\frac{D}{\sqrt{\underline{L}}\sqrt{n+1}}$.
\end{theorem}

Next to such residual estimates, we also prove the following result on effective rates for the sequence under a strong convexity assumption:

\begin{theorem}\label{SBrates}
Let $(E,\mathcal{E},\mu)$ and $(\Omega,\mathcal{F},\PP)$ be probability spaces, with $(E,\mathcal{E},\mu)$ complete, and let $X$ be a separable Hadamard space with the geodesic extension property and at least two points. Let $C\subseteq X$ be closed, convex and nonempty. Let $f:E\times C\to \mathbb{R}$ be a function with properties  \eqref{A1} -- \eqref{A3}  as above and assume additionally that $f(e,\cdot)$ is strongly convex with parameter $\alpha(e)>0$, i.e.\
\[
f(e,\gamma(tl))\leq (1-t)f(e,\gamma(0))+tf(e,\gamma(l)) -t(1-t)\frac{\alpha(e)}{2}d^2(\gamma(0),\gamma(l))
\]
for any geodesic $\gamma:[0,l]\to C$, any $t\in [0,1]$ and any $e\in E$, where $\alpha$ is integrable with $\underline{\alpha}:=\int \alpha\,d\mu>0$. Define $F(x):=\int f(e,x)\,d\mu(e)$ and assume $\mathrm{argmin}F\neq\emptyset$. Let $(x_n)$ be the iteration given by \eqref{SB} with $(t_n)$ and $(\zeta_{n+1})$, assume \eqref{parameters}, and write $\underline{L}:=\int L^2\,d\mu$.

Then $(x_n)$ a.s.\ and in mean strongly converges to the unique minimizer $x^*$ of $F$. Moreover, the following rates of convergence apply: Let $\chi:(0,\infty)\to\mathbb{N}$ and $\theta:\mathbb{N}\times (0,\infty)\to\mathbb{N}$ be such that 
\[
\forall \varepsilon>0\left(\sum_{n=\chi(\varepsilon)}^\infty t_n^2<\varepsilon\right) \text{ and }\forall b>0\ \forall k\in\mathbb{N}\left(\sum_{n=k}^{\theta(k,b)}t_n\geq b\right).
\]
Let $T > \sum_{n=0}^\infty t_n^2$. Lastly, let $b>0$ be such that $b>d^2(x_0,x^*)$. Then
\[
\forall \varepsilon>0\ \forall n\geq \rho(\varepsilon)\left(\EE[d^2(x_n,x^*)]<\varepsilon\right)
\]
with rate $\rho(\varepsilon):=\theta(\chi(\varepsilon/2\underline{L}),8(b+\underline{L}T)/\varepsilon\underline{\alpha})$ and 
\[
\forall \lambda,\varepsilon>0\left(\PP\left(\exists n\geq\rho'(\lambda,\varepsilon)\left(d^2(x_n,x^*)\geq\varepsilon\right)\right)<\lambda\right)
\]
with rate $\rho'(\lambda,\varepsilon):=\rho(\lambda\varepsilon)$.
\end{theorem}

The proof essentially relies on a similar approach as the results given by Neri, Powell and the author in \cite{NeriPischkePowell2025} for the quantitative asymptotic behavior for general stochastic processes subscribing to a Fej\'er monotonicity condition under suitable uniqueness assumptions.\footnote{The results from \cite{NeriPischkePowell2025}, and likewise the present results regarding rates of convergence, have been obtained using the logic-based methodology of \emph{proof mining} \cite{Kohlenbach2008,Kohlenbach2019}, and are part of a series of recent applications of these methods to probability theory and stochastic optimization  \cite{NeriPischke2024,NeriPischkePowell2025,NeriPowell2025,NeriPowell2024,PischkePowell2025,PischkePowell2026}. As common in proof mining however, this paper avoids any reference to mathematical logic.}

Indeed, the above result can in fact be substantially extended to other regularity conditions on the mean function $F(x):=\int f(e,x)\,d\mu(e)$ beyond strong convexity, including generalized weak sharp minima, as discussed in more detail in recent work by the author and Powell \cite{PischkePowell2026} on the general theory of such regularity conditions and their relation to quantitative convergence analyses for stochastic processes, extending \cite{NeriPischkePowell2025}. That however requires various additional machinery beyond that used in \cite{NeriPischkePowell2025} or the present paper.

\section{Preliminaries}\label{preliminaries}

We now discuss the few preliminary definitions, results and notations that we require throughout. As mentioned in the introduction, beyond the results indicated in this paper we refer to \cite{AlexanderKapovitchPetrunin2023,Bacak2014a,BridsonHaefliger1999} for a comprehensive overview of geodesic metric spaces and their properties, in particular to \cite{Bacak2014a} for aspects of (stochastic) optimization. Further, beyond the results indicated, we refer to e.g.\ \cite{Klenke2020} for a standard textbook on probability theory.

Let $(X,d)$ be a metric space. A geodesic is an isometry $\gamma:[0,l]\to X$. We say that it joins $x=\gamma(0)$ and $y=\gamma(l)$ (where necessarily $l=d(x,y)$). $X$ is called (uniquely) geodesic if every two points are joined by a (unique) geodesic. We call the image of a geodesic such as the above a geodesic segment, and in the uniquely geodesic case denote it by $[x,y]$. A geodesic ray is an isometry $r:[0,\infty)\to X$, and we say that $r$ issues from $r(0)$. A space $X$ has the geodesic extension property if for all $x\neq y\in X$, there is a ray $r:[0,\infty)\to X$ issuing from $x$ such that $r(t)=y$ for some $t>0$. 

A geodesic metric space $(X,d)$ is called a $\CAT$ space (also called a space of nonpositive curvature in the sense of Alexandrov) if it satisfies 
\[
d^2(\gamma(tl),x)\leq (1-t)d^2(\gamma(0),x)+td^2(\gamma(l),x)-t(1-t)d^2(\gamma(0),\gamma(l))
\]
for all $x\in X$ and $t\in [0,1]$ and all geodesics $\gamma:[0,l]\to X$ (that is, an extension of the so-called Bruhat-Tits $\mathrm{CN}$-inequality \cite{BruhatTits1972} to geodesics). Any $\CAT$ space is uniquely geodesic and a complete $\CAT$ space is called a Hadamard space. Further, a $\CAT$ space has nonpositive curvature in the sense of Busemann (see e.g.\ \cite{Bacak2014a,BridsonHaefliger1999}) and hence the metric is jointly convex in the sense that $d(\gamma(tl),\eta(tm))$ is convex on $t\in [0,1]$ for all geodesics $\gamma:[0,l]\to X$ and $\eta:[0,m]\to X$ (see e.g.\ Proposition 1.1.5 in \cite{Bacak2014a}).

Weak convergence in $\CAT$ spaces goes back to the work of Jost \cite{Jost1994} and is often called $\Delta$-convergence following the work of Kirk and Panyanak \cite{KirkPanyanak2008} (we refer in particular to the discussion in \cite{Bacak2013} on that matter). We define weak convergence here as follows (see e.g.\ \cite{Bacak2014a}):

Given a bounded sequence $(x_n)\subseteq X$ and a point $x\in X$, their asymptotic radius is given by
\[
r(x_n,x):=\limsup_{n\to\infty}d^2(x_n,x)
\]
and the general asymptotic radius of the sequence $(x_n)$ is given by
\[
r(x_n):=\inf_{x\in X}r(x_n,x).
\]
A point $x\in X$ is called an asymptotic center of $(x_n)$ if $r(x_n,x)=r(x_n)$. In Hadamard spaces, asymptotic centers exist and are unique (see e.g.\ Proposition 7 in \cite{DhompongsaKirkSims2006}). We say that a bounded sequence $(x_n)$ weakly converges to $x\in X$, written $x_n\to^w x$, if $x$ is the asymptotic center of each subsequence of $(x_n)$. A point $x\in X$ is a weak cluster point of $(x_n)$ if there is a subsequence $(x_{n_k})$ of $(x_n)$ with $x_{n_k} \to^w x$. Each bounded sequence has a weak cluster point (see e.g.\ Proposition 3.1.2 in \cite{Bacak2014a}).

We write $\mathfrak{W}(x_n)$ for the set of all weak cluster points of $(x_n)$ and $\mathfrak{S}(x_n)$ for the set of all strong cluster points of $(x_n)$, the latter defined as usual using the metric.

A set $C\subseteq X$ is called convex if the geodesics between any two points in $C$ are contained in $C$, and we call a function $g:C\to \mathbb{R}$ convex if
\[
g(\gamma(tl))\leq (1-t)g(\gamma(0))+tg(\gamma(l))
\]
for any geodesic $\gamma:[0,l]\to C$ and any $t\in [0,1]$. Further, we call $g$ lower-semicontinuous (lsc) if
\[
\liminf_{n\to\infty} g(x_n)\geq g(x)
\]
whenever $(x_n)\subseteq C$ is a sequence such that $x_n\to x$. Crucially, every convex lsc function on a Hadamard space is also weakly lower-semicontinuous (weakly lsc, see Lemma 3.1 in \cite{Bacak2013}), i.e.\ it satisfies the above inequality even if $(x_n)\subseteq C$ is a sequence such that $x_n\to^w x$. A particular example of a convex function on a Hadamard space is the Busemann function $b_r:X\to\mathbb{R}$ associated to a ray $r$, defined by
\[
b_r(x):=\lim_{t\to\infty} (d(x,r(t))-t),
\]
which is nonexpansive (1-Lipschitz), convex and satisfies $b_r(r(0))=0$ (see Example 2.2.10 in \cite{Bacak2014a}).

Throughout this paper, if not stated otherwise, we let $(X,d)$ be a separable Hadamard space and $(\Omega,\mathcal{F},\PP)$ as well as $(E,\mathcal{E},\mu)$ be two probability spaces, with $(E,\mathcal{E},\mu)$ complete. All probabilistic notions such as measurability, random variables, almost sureness (a.s.), expectation, etc., are understood relative to the space $(\Omega,\mathcal{F},\PP)$, if not stated otherwise. In particular, an $X$-valued random variable is a map $x:\Omega\to X$ which is measurable relative to $\mathcal{F}$ and the Borel $\sigma$-algebra $\mathcal{B}(X)$ of that space. We denote (conditional) expectations over $(\Omega,\mathcal{F},\PP)$ by $\EE$. All properties as well as (in-)equalities between random variables are understood to hold only almost surely, if not stated otherwise. Sometimes we are working on general measurable spaces $(T,\mathcal{T})$ or even measure spaces $(T,\mathcal{T},\tau)$. In such cases, we are always quite explicit regarding measurability notions, and in particular use the expression ``almost everywhere'' (a.e.) instead. 

\section{A nonlinear variant of Pettis' theorem}

In this section, we prove the following measurability result for weak limits in separable Hadamard spaces:

\begin{proposition}\label{nonlinearPettis}
Let $(T,\mathcal{T},\tau)$ be a $\sigma$-finite measure space and let $X$ be a separable Hadamard space. Let $(x_n)$ be a sequence of $\mathcal{T}$/$\mathcal{B}(X)$-measurable functions and let $x:T\to X$ be a function such that the $(x_n)$ are bounded and $x_n\to^w x$ almost everywhere. Then there is a $\mathcal{T}$/$\mathcal{B}(X)$-measurable $y$ which is equal to $x$ almost everywhere. If $(T,\mathcal{T},\tau)$ is complete, then $x$ itself is measurable.
\end{proposition}

In separable Hilbert spaces, the above result is a consequence (see Corollary 1.13 in \cite{Pettis1938}) of the seminal theorem of Pettis on weakly measurable functions (see Theorem 1.1 in \cite{Pettis1938}).

Our proof is rather simple, relying only on a few results from measurable selection theory. We generally refer to \cite{AliprantisBorder2006,AubinFrankowska2009,CastaingValadier1977} for further background on that area. Given measurable spaces $(T,\mathcal{T})$ and $(X,\mathcal{X})$, call a set-valued map $\varphi:T\to 2^X$ graph measurable if
\[
\mathrm{gra}(\varphi):=\{(t,x)\in T\times X\mid x\in\varphi(t)\}\in\mathcal{T}\otimes\mathcal{X}.
\]
Over complete $\sigma$-finite measure spaces $(T,\mathcal{T},\tau)$ and complete separable metric spaces $X$, and if $\varphi$ has nonempty closed images, this is equivalent to the (weak) measurability of $\varphi:T\to 2^X$, that is that 
\[
\varphi^{-1}(C):=\{t\in T\mid \varphi(t)\cap C\neq\emptyset\}\in\mathcal{T}
\]
for all open sets $C\subseteq X$ (see e.g.\ Theorem 8.1.4 in \cite{AubinFrankowska2009}).

The key result we need is the following on the (graph) measurability of minimizing maps.

\begin{lemma}\label{minMeas}
Let $(T,\mathcal{T})$ be a measurable space and let $X$ be a complete separable metric space. Further, let $g:T\times X\to\mathbb{R}$ be a Carath\'eodory function.\footnote{Given topological spaces $X$, $Y$ and a measurable space $(T,\mathcal{T})$, a function $g:T\times X\to Y$ is called a Carath\'eodory map if $g(\cdot,x)$ is measurable for all $x\in X$ and $g(t,\cdot)$ is continuous for all $t\in T$.} Then the functions 
\[
m(t):=\inf_{x\in X}g(t,x)\text{ and }M(t):=\mathrm{argmin}_{x\in X}g(t,x)
\]
are measurable and graph measurable, respectively. If $\mathcal{T}$ is complete w.r.t.\ a $\sigma$-finite measure $\tau$, then $M$ is (weakly) measurable.
\end{lemma}

Lemma \ref{minMeas} commonly appears in the literature only under the assumption that the measurable space is complete and $\sigma$-finite for some measure, with the conclusion of (weak) measurability of $M$ (see e.g.\ Theorem 8.2.11 in \cite{AubinFrankowska2009}). While not strictly necessary for the above nonlinear variant of Pettis' theorem, we will later rely on the above version when working over certain non-complete spaces and so already here dispense of that completeness assumption. Rederiving the above result in such a generalized context requires some care on measurability and on the assumptions.

Concretely, to prove Lemma \ref{minMeas}, we first need the following result on Carath\'eodory functions:
 
\begin{lemma}[Corollary 18.8 in \cite{AliprantisBorder2006}]\label{eqMeas}
Let $g:T\times X\to (-\infty,+\infty]$ be a Carath\'eodory function. Then
\[
\varphi(t):=\{x\in X\mid g(t,x)=0\}
\]
is graph measurable.
\end{lemma}

We can now provide the proof of Lemma \ref{minMeas}:

\begin{proof}[Proof of Lemma \ref{minMeas}]
Using that $X$ is separable, fix a countable dense set $(z_n)$. Then 
\[
m(t)=\inf_{x\in X}g(t,x)=\inf_{n\in\mathbb{N}}g(t,z_n)
\]
by continuity of $g(t,\cdot)$. Thus, $m$ is measurable as every $g(\cdot,z_n)$ is measurable. As $g$ is a Carath\'eodory function, so is $g'(t,x):=g(t,x)-m(t)$. Using Lemma \ref{eqMeas}, we then get that
\[
M(t)=\{y\in X\mid g(t,y)= m(t)\}=\{y\in X\mid g'(t,y)=0\}
\]
is graph measurable. For the (weak) measurability of $M$ if $\mathcal{T}$ is complete w.r.t.\ a $\sigma$-finite measure $\tau$, simply appeal to Theorem 8.2.11 (or Theorem 8.1.4) in \cite{AubinFrankowska2009}.
\end{proof}

This result will then be combined with the Kuratowski–Ryll-Nardzewski selection theorem:

\begin{lemma}[see e.g.\ Theorem 8.1.3 in \cite{AubinFrankowska2009}]\label{Selection}
Let $X$ be a complete separable metric space and let $(T,\mathcal{T})$ be a measurable space. If $\varphi:T\to 2^X$ is a (weakly) measurable map such that $\varphi(t)$ is nonempty and closed for any $t\in T$, then there exists a measurable function $x: T\to X$ such that $x(t)\in\varphi(t)$ for all $t\in T$.
\end{lemma}

As already highlighted before, in the present paper we want to be explicit regarding completeness assumptions on the various involved probability spaces, and will assume it only if it is really necessary for a result. As Proposition \ref{nonlinearPettis} in particular will later only be needed to infer the existence of a measurable function equal to the weak limit almost everywhere, we have chosen to stress a corresponding variant above that dispenses of the usual completeness assumption in that case. The result that makes this non-complete case of Proposition \ref{nonlinearPettis} possible at the full generality of $\sigma$-finite measure spaces is the following lemma (which we label as folklore, see e.g.\ p.\ 103 in \cite{Kechris1995}):

\begin{lemma}[folklore]\label{removeCompleteness}
Let $(T,\mathcal{T},\tau)$ be a measure space and let $(X,\mathcal{X})$ be a measurable space such that $\mathcal{X}$ is countably generated, i.e.\ $\mathcal{X}=\sigma(\mathcal{X}_0)$ for a countable set $\mathcal{X}_0\subseteq 2^X$. Let $x:T\to X$ be $\overline{\mathcal{T}}$/$\mathcal{X}$-measurable, where $\overline{\mathcal{T}}$ is the completion of $\mathcal{T}$ w.r.t.\ $\tau$. Then there exists a $\mathcal{T}$/$\mathcal{X}$-measurable $y:T\to X$ such that $x=y$ almost everywhere.
\end{lemma}

We can now establish our nonlinear variant of Pettis's theorem, the key observation being that $r(x_n(t),y)$ defines a Carath\'eodory function:

\begin{proof}[Proof of Proposition \ref{nonlinearPettis}]
Assume first that $(T,\mathcal{T},\tau)$ is complete. Fix a set $T_0$ of measure zero such that $(x_n(t))$ is bounded and $x_n(t)\to^w x(t)$ for all $t\in T_0^c$. First note that $(t,y)\mapsto r(x_n(t),y)$ for $(t,y)\in T_0^c\times X$ is a Carath\'eodory function. Indeed, recall that by definition $r(x_n(t),y)=\limsup_{n\to\infty}d^2(x_n(t),y)$. First fix $t\in T_0^c$. As each $d^2(x_n(t),\cdot)$ is locally Lipschitz continuous with the same constant, it follows that $r(x_n(t),\cdot)$ is locally Lipschitz and hence continuous (see e.g.\ Example 2.2.8 in \cite{Bacak2014a}). Now fix $y\in X$. As each $d^2(x_n(\cdot),y)$ is measurable, we immediately get that $r(x_n(\cdot),y)$ is measurable.

Lemma \ref{minMeas} then yields that
\[
\varphi'(t):=\mathrm{argmin}_{y\in X}r(x_n(t),y),
\]
defined for $t\in T_0^c$, is graph measurable. As $(x_n(t))$ is bounded, we have by the existence of asymptotic centers that $\varphi'(t)$ is nonempty for all $t\in T_0^c$. Further, $\varphi'(t)$ is closed by continuity of $r(x_n(t),\cdot)$. Now define $\varphi(t):=\varphi'(t)$ for $t\in T_0^c$ and $\varphi(t):=X$ otherwise. Using that $\varphi'$ is graph measurable, we get
\[
\mathrm{gra}(\varphi)=\mathrm{gra}(\varphi')\cup (T_0\times X)\in\mathcal{T}\otimes\mathcal{B}(X),
\]
so that $\varphi$ is still graph measurable, with $\varphi(t)$ nonempty and closed for all $t\in T$. Hence, using that $(T,\mathcal{T},\tau)$ is complete, $\varphi$ is (weakly) measurable (recall Theorem 8.1.4 in \cite{AubinFrankowska2009}). Using Lemma \ref{Selection}, there exists a measurable function $y$ such that $y(t)\in\varphi(t)$ for all $t\in T$. For $t\in T_0^c$, by the uniqueness of asymptotic centers, we get that $x(t)=y(t)$. Thus $x=y$ almost everywhere. As the space $(T,\mathcal{T},\tau)$ is complete, we get that $x$ is measurable as well.

If $(T,\mathcal{T},\tau)$ is not complete, consider its completion $(T,\overline{\mathcal{T}},\overline{\tau})$. The above yields that $x$ is $\overline{\mathcal{T}}$/$\mathcal{B}(X)$-measurable, so that Lemma \ref{removeCompleteness} yields a $\mathcal{T}$/$\mathcal{B}(X)$-measurable $y$ which is equal to $x$ almost everywhere (note that $\mathcal{B}(X)$ is countably generated since $X$ is a complete separable metric space).
\end{proof}

\section{Stochastic quasi-Fej\'er monotonicity and weak (ergodic) convergence}

As outlined in the introduction, the main technical convergence result we present is a general proposition on the weak and strong convergence of stochastic quasi-Fej\'er monotone sequences in metric spaces.

For that, we first introduce some convenient notation: Given a filtration $\mathsf{F}=(\mathsf{F}_n)$ of $\mathcal{F}$, that is a sequence of sub-$\sigma$-algebras of $\mathcal{F}$ such that $\mathsf{F}_n\subseteq\mathsf{F}_m$ for $n\leq m$, we write $\ell_+(\mathsf{F})$ for the set of sequences of non-negative real-valued random variables $(e_n)$ that are adapted to the filtration, i.e.\ where $e_n$ is $\mathsf{F}_n$-measurable for all $n\in\mathbb{N}$. Further, we write $\ell^1_+(\mathsf{F})$ for the set of all $(e_n)\in \ell_+(\mathsf{F})$ such that $\sum_{n\in\mathbb{N}}e_n<+\infty$ a.s.

Now, fixing a solution set $Z\subseteq X$ and a filtration $\mathsf{F}=(\mathsf{F}_n)$, stochastic quasi-Fej\'er monotonicity for a sequence $(x_n)$ adapted to $\mathsf{F}$ in our context takes the form of requiring
\[
\EE[\phi(d(x_{n+1},z))\mid \mathsf{F}_n]\leq (1+\chi_n(z)) \phi(d(x_n,z))-\theta_n(z)+\eta_n(z)\text{ a.s.}
\]
for all $z\in Z$ and $n\in\mathbb{N}$, where $\phi$ is a suitable perturbation function and $(\chi_n(z)),(\eta_n(z))\in\ell^1_+(\mathsf{F})$ as well as $(\theta_n(z))\in\ell_+(\mathsf{F})$ are error sequences which might depend on the point $z$ in question.

As such, our notion is a direct lift of the rather general stochastic quasi-Fej\'er monotonicity considered in separable Hilbert spaces in the seminal work of Combettes and Pesquet \cite{CombettesPesquet2015} (see equation (2.5) therein and recall the further references given in the introduction), to the metric setting.

We can then obtain the following result on weak and strong convergence based on weak and strong cluster points, which is itself an extension of a corresponding result in separable Hilbert spaces given by Combettes and Pesquet \cite{CombettesPesquet2015} (see Proposition 2.3 therein). As mentioned in the introduction, as such our result is in particular an immediate synthesis of Proposition 2.3 from \cite{CombettesPesquet2015} (especially concerning items (1) -- (3) of the following Proposition \ref{weakConvergence}, which are highly derivative of it and essentially are a direct lift to the metric setting) and Proposition 3.3 from \cite{BacakSearstonSims2012} on weak convergence of Fej\'er monotone sequences in Hadamard spaces. Indeed, the key aspect of the convergence of Fej\'er monotone sequences in Hadamard spaces is the following result essentially derived (in the proof of) Proposition 3.3 in \cite{BacakSearstonSims2012}, and it immediately allows us to derive items (4) and (5) of the following Proposition \ref{weakConvergence}. For self-containedness, we rederive it here, with a slightly different argument.

\begin{lemma}[essentially Proposition 3.3 in \cite{BacakSearstonSims2012}]\label{nonlinearFejer}
Let $X$ be a Hadamard space, let $Z\subseteq X$ be nonempty and closed, and let $(x_n)\subseteq X$ be a given sequence such that $d(x_n,z)$ converges for all $z\in Z$.
\begin{enumerate}
\item If $\mathfrak{W}(x_n)\subseteq Z$, then $(x_n)$ weakly converges to some point in $Z$.
\item If $\mathfrak{S}(x_n)\cap Z\neq\emptyset$, then $(x_n)$ strongly converges to some point in $Z$.
\end{enumerate}
\end{lemma}
\begin{proof}
Note that $(x_n)$ is bounded and therefore has a weak cluster point. For the first item, it hence suffices to show that $(x_n)$ has a unique weak cluster point. To that effect, let $(x_{n_k})$ and $(x_{m_k})$ be subsequences of $(x_n)$ with asymptotic centers $c_1$ and $c_2$, respectively. As $\mathfrak{W}(x_n)\subseteq Z$, we have $c_1,c_2\in Z$. Therefore $d(x_n,c_1)$ and $d(x_n,c_2)$ and so also $d^2(x_n,c_1)$ and $d^2(x_n,c_2)$ converge. Now, assume w.l.o.g.\ that $r(x_{n_k})\leq r(x_{m_k})$. We then have
\begin{align*}
r(x_{m_k})\geq r(x_{n_k})&=r(x_{n_k},c_1)\\
&=\limsup_{k\to\infty} d^2(x_{n_k},c_1)\\
&=\lim_{n\to\infty} d^2(x_n,c_1)\\
&=\limsup_{k\to\infty} d^2(x_{m_k},c_1)\\
&=r(x_{m_k},c_1)\geq r(x_{m_k}),
\end{align*}
where the third (and fourth) line follows from the fact that $d^2(x_n,c_1)$ (and so every subsequence of it) converges. Thus in particular $r(x_{m_k})=r(x_{m_k},c_1)$ and so $c_1$ is also an asymptotic center of $(x_{m_k})$, next to $c_2$. As asymptotic centers are unique, we get $c_1=c_2$ which completes the argument. 

For the strong convergence result, assume $\mathfrak{S}(x_n)\cap Z\neq\emptyset$. Concretely, let $x\in Z$ be a strong accumulation point of $(x_n)$, i.e.\
\[
\liminf_{n\to\infty}d(x_n,x)=0.
\]
As $d(x_n,x)$ converges, we get $\lim_{n\to\infty}d(x_n,x)=0$, which completes the proof.
\end{proof}

Like in \cite{CombettesPesquet2015}, and any other result on stochastic quasi-Fej\'er monotonicity for that matter, the key statistical ingredient of our convergence result is the seminal Robbins-Siegmund theorem on almost-supermartingale convergence:

\begin{lemma}[Theorem 1 in \cite{RobbinsSiegmund1971}]\label{RS}
Let $\mathsf{F}=(\mathsf{F}_n)$ be a filtration, and let $(\alpha_n),(\theta_n)\in\ell_+(\mathsf{F})$ as well as $(\eta_n),(\chi_n)\in\ell^1_+(\mathsf{F})$ be such that
\[
\EE[\alpha_{n+1}\mid \mathsf{F}_n]\leq (1+\chi_n)\alpha_n-\theta_n+\eta_n\text{ a.s.}
\] 
for any $n\in\mathbb{N}$. Then $(\alpha_n)$ a.s.\ converges to a nonnegative real-valued random variable and $(\theta_n)\in\ell^1_+(\mathsf{F})$.
\end{lemma}

\begin{proposition}[extending Proposition 2.3 in \cite{CombettesPesquet2015} and Proposition 3.3 in \cite{BacakSearstonSims2012}]\label{weakConvergence}
Let $X$ be a separable Hadamard space, let $Z\subseteq X$ be nonempty and closed, and let $\phi:[0,+\infty)\to [0,+\infty)$ be strictly increasing such that $\lim_{t\to +\infty}\phi(t)=+\infty$. Let $\mathsf{F}=(\mathsf{F}_n)$ be a filtration and let $(x_n)$ be a sequence of $X$-valued random variables adapted to $\mathsf{F}$ such that it is stochastically quasi-Fej\'er monotone w.r.t.\ $Z$, that is for any $z\in Z$ there are $(\chi_n(z)),(\eta_n(z))\in\ell^1_+(\mathsf{F})$ and $(\theta_n(z))\in\ell_+(\mathsf{F})$ such that for all $n\in\mathbb{N}$:
\[
\EE[\phi(d(x_{n+1},z))\mid \mathsf{F}_n]\leq (1+\chi_n(z)) \phi(d(x_n,z))-\theta_n(z)+\eta_n(z)\text{ a.s.}\tag{$*$}\label{randFejer}
\]
Then we have the following assertions:
\begin{enumerate}
\item $\sum_{n\in\mathbb{N}}\theta_n(z)<+\infty$ a.s.\ for all $z\in Z$.
\item $(x_n)$ is bounded a.s.
\item There exists a set $\widetilde{\Omega}$ with $\PP(\widetilde{\Omega})=1$ such that for all $\omega\in\widetilde{\Omega}$ and $z\in Z$, the sequence given by $d(x_n(\omega),z)$ converges.
\item If $\mathfrak{W}(x_n)\subseteq Z$ a.s., then $(x_n)$ weakly converges a.s.\ to a $Z$-valued random variable.
\item If $\mathfrak{S}(x_n)\cap Z\neq\emptyset$ a.s., then $(x_n)$ strongly converges a.s.\ to a $Z$-valued random variable.
\end{enumerate}
\end{proposition}
\begin{proof}
Ad (1): For any $z\in Z$, the above inequality \eqref{randFejer} and Lemma \ref{RS} immediately yield $\sum_{n\in\mathbb{N}}\theta_n(z)<+\infty$ a.s.\\

Ad (2): Let $z\in Z$ be arbitrary. The above inequality \eqref{randFejer} and Lemma \ref{RS} immediately yield that $\phi(d(x_n,z))$ converges a.s.\ to some nonnegative real-valued random variable $\alpha_z$. As $\lim_{t\to +\infty}\phi(t)=+\infty$, we get that the sequence $d(x_n,z)$ is bounded a.s., which is hence true also for $(x_n)$.\\

Ad (3): Let us first note that given any $z\in Z$, $d(x_n,z)$ converges a.s., namely at every point where the sequence $\phi(d(x_n,z))$ converges. For suppose not, then there is a point $\omega\in\Omega$ such that the sequence $\phi(d(x_n(\omega),z))$ converges and hence, by the argument from item (2), such that $d(x_n(\omega),z)$ is bounded and where there are subsequences $d(x_{n_k}(\omega),z)$ and $d(x_{m_k}(\omega),z)$ with respective distinct limits $\tau_1(\omega)$ and $\tau_2(\omega)$. W.l.o.g. suppose $\tau_2(\omega)>\tau_1(\omega)$ and take $\varepsilon(\omega)\in (0,(\tau_2(\omega)-\tau_1(\omega))/2)$. For large enough $n$, we will have
\[
d(x_{n_k}(\omega),z)\leq \tau_1(\omega)+\varepsilon(\omega)<\tau_2(\omega)-\varepsilon(\omega)\leq d(x_{m_k}(\omega),z).
\]
As $\phi$ is strictly increasing, we immediately get 
\[
\phi(d(x_{n_k}(\omega),z))\leq \phi(\tau_1(\omega)+\varepsilon(\omega))<\phi(\tau_2(\omega)-\varepsilon(\omega))\leq \phi(d(x_{m_k}(\omega),z)).
\]
Taking the limit as $n\to \infty$ gives
\[
\alpha_z(\omega)\leq \phi(\tau_1(\omega)+\varepsilon(\omega))<\phi(\tau_2(\omega)-\varepsilon(\omega))\leq \alpha_z(\omega)
\]
for the random variable $\alpha_z$ from item (2), which is a contradiction. 

Now, we move on to actually showing item (3). As $X$ is separable, also $Z$ is separable. Fix a countable dense set $Z_0$ in $Z$. By the previous, for every $z\in Z_0$ there exists a set $\Omega_z$ of measure one such that $d(x_n(\omega),z)$ converges for any $\omega\in \Omega_z$. Define $\widetilde{\Omega}:=\bigcap_{z\in Z_0}\Omega_z$. Clearly $\widetilde{\Omega}$ has measure one. Let now $z\in Z$ and $\omega\in \widetilde{\Omega}$ be arbitrary. As $Z_0$ is dense in $Z$, there exists a sequence $(z_k)$ in $Z_0$ such that $z_k\to z$. As $\omega\in \widetilde{\Omega}\subseteq\Omega_{z_k}$, we in particular have that $d(x_n(\omega),z_k)$ converges for any $k\in\mathbb{N}$ and for any $\omega\in \widetilde{\Omega}$, say with limit $\tau_k(\omega)$. For any $k\in\mathbb{N}$ and any such $\omega\in \widetilde{\Omega}$, we thus have 
\[
-d(z_k,z)\leq d(x_n(\omega),z)-d(x_n(\omega),z_k)\leq d(z_k,z)
\]
using triangle inequality, so that we get
\begin{align*}
-d(z_k,z)&\leq \liminf_{n\to\infty}d(x_n(\omega),z)-\lim_{n\to\infty}d(x_n(\omega),z_k)\\
&=\liminf_{n\to\infty}d(x_n(\omega),z)-\tau_k(\omega)\\
&\leq \limsup_{n\to\infty}d(x_n(\omega),z)-\tau_k(\omega)\\
&= \limsup_{n\to\infty}d(x_n(\omega),z)-\lim_{n\to\infty}d(x_n(\omega),z_k)\\
&\leq d(z_k,z).
\end{align*}
Taking the limit as $k\to\infty$ yields that $d(x_n(\omega),z)$ converges with
\[
\lim_{n\to\infty}d(x_n(\omega),z)=\lim_{k\to\infty}\tau_k(\omega).
\]

Ad (4): Let $\widetilde{\Omega}$ be the set with measure one from item (3) and suppose that $\mathfrak{W}(x_n)\subseteq Z$ a.s., say $\mathfrak{W}(x_n(\omega))\subseteq Z$ for all $\omega\in\widehat{\Omega}$ for $\widehat{\Omega}$ with measure one. Lemma \ref{nonlinearFejer} now yields that for any $\omega\in \widehat{\Omega}\cap \widetilde{\Omega}$, the sequence $(x_n(\omega))$ converges weakly to an element $x(\omega)$ of $Z$. This then yields the claim as follows: $\widehat{\Omega}\cap \widetilde{\Omega}$ still has measure one. By Proposition \ref{nonlinearPettis}, we get that there is a measurable $y$ with $x=y$ a.s. Thus $y$ is a.s.\ $Z$-valued and so w.l.o.g.\ we can assume that $y$ is $Z$-valued. Clearly, $(x_n)$ converges weakly to $y$ almost surely.\\

Ad (5): In a similar way as item (4) above, using in particular item (3), Lemma \ref{nonlinearFejer} now yields that $x_n(\omega)$ strongly converges to some point $x(\omega)\in Z$ on a set of measure one. The measurability of $x$ follows immediately as $X$ is separable.
\end{proof}

Under suitable restrictions on the space, we can now also provide a stochastic weak ergodic convergence result. We establish this result not for all separable Hadamard spaces, but only for those that satisfy the condition $(\overline{Q}_4)$, i.e.\ for points $x,y,p,q\in X$ and any point $m\in [x,y]$:
\[
d(p,x)\leq d(x,q)\text{ and }d(p,y)\leq d(y,q)\text{ implies }d(p,m)\leq d(m,q).
\]
Equivalently, and more crucially for the present paper, the condition expresses that
\[
F(p,q):=\{x\in X\mid d(p,x)\leq d(x,q)\}\text{ is convex}\tag{$\overline{Q}_4$}
\]
for all $p,q\in X$. This condition was introduced by Kakavandi in \cite{Kakavandi2013} for studying topologies capturing weak convergence (which was recently, and by quite a different method, substantially extended by Lytchak and Petrunin \cite{LytchakPetrunin2023}) and is a slight strengthening of the condition $(Q_4)$ introduced by Kirk and Panyanak in \cite{KirkPanyanak2008} for establishing convergence of midpoint sequences in geodesic segments (which was subsequently weakened in \cite{EspinolaFernandezLeon2009}). Examples of spaces that satisfy $(\overline{Q}_4)$ include Hilbert spaces, $\mathbb{R}$-trees, Hadamard spaces with constant curvature and their closed convex
subsets (see \cite{Kakavandi2013} and see also \cite{EspinolaFernandezLeon2009} for related results for the condition $(Q_4)$). However, not all Hadamard spaces even satisfy the condition $(Q_4)$ as shown in \cite{EspinolaFernandezLeon2009}. In particular, any $\mathrm{CAT}(0)$ gluing space containing two spaces with constant but different curvatures does not.

In the recent paper \cite{KhatibzadehMoosavi2023}, this condition was used to establish the weak ergodic convergence of a nonlinear variant, first studied by Ba\v{c}\'ak \cite{Bacak2014b}, of the proximal splitting method originally considered by Passty \cite{Passty1979}.

While presented as tailored to that particular iteration, the approach given in \cite{KhatibzadehMoosavi2023} essentially establishes the following general convergence principle, which can be seen as a nonlinear variant of Passty's ergodic convergence theorem given as Lemma 1 in \cite{Passty1979}. As the result only appears implicitly in \cite{KhatibzadehMoosavi2023}, we reprove it here, in particular simplifying parts of the argument used in \cite{KhatibzadehMoosavi2023}:

\begin{lemma}[essentially the proof of Theorem 4.2 in \cite{KhatibzadehMoosavi2023}]\label{nonlinearPassty}
Let $X$ be a Hadamard space satisfying $(\overline{Q}_4)$, let $Z\subseteq X$ be nonempty, closed and convex, and let $(x_n)\subseteq X$ be a given sequence such that $d(x_n,z)$ converges for all $z\in Z$. Further, let $(t_n)\subseteq (0,+\infty)$ be such that $\sum_{n\in\mathbb{N}}t_n=+\infty$ and let $(\overline{x}_n)$ be the associated sequence of weighted averages. If $\mathfrak{W}(\overline{x}_n)\subseteq Z$, then $(\overline{x}_n)$ weakly converges to some point in $Z$.
\end{lemma}
\begin{proof}
The sequence $(x_n)$ and hence also $(\overline{x}_n)$ is bounded and therefore has a weak cluster point. It thus suffices to show that $(\overline{x}_n)$ has no further weak cluster point. Suppose that $c_1,c_2$ are two weak cluster points and let $(\overline{x}_{n_k})$ and $(\overline{x}_{m_k})$ be subsequences of $(\overline{x}_n)$ with asymptotic centers $c_1$ and $c_2$, respectively. As $\mathfrak{W}(\overline{x}_n)\subseteq Z$, we have $c_1,c_2\in Z$. Therefore $d(x_n,c_1)$ and $d(x_n,c_2)$ and so also $d^2(x_n,c_1)$ and $d^2(x_n,c_2)$ converge. Assume w.l.o.g.\ that $\lim_{n\to\infty}d^2(x_n,c_1)\leq \lim_{n\to\infty}d^2(x_n,c_2)$. We now distinguish two cases, namely $\lim_{n\to\infty}d^2(x_n,c_1)< \lim_{n\to\infty}d^2(x_n,c_2)$ and $\lim_{n\to\infty}d^2(x_n,c_1)= \lim_{n\to\infty}d^2(x_n,c_2)$.

We begin with the former. There hence exists an $n_0\geq 1$ with $d^2(x_n,c_1)< d^2(x_n,c_2)$ and so in particular $x_n\in F(c_1,c_2)$ for all $n\geq n_0$. Given $i\in\mathbb{N}$, denote the weighted average of $(x_{n+i})$ with $(t_n)$ by $\overline{x}^{i}_{m}$ for $m\geq i$, that is $\overline{x}^{i}_{i}:=x_i$ and $\overline{x}^{i}_{m+1}:=\frac{\overline{t}_m}{\overline{t}_{m+1}}\overline{x}^{i}_m\oplus \frac{t_{m+1}}{\overline{t}_{m+1}}x_{m+1}$ for $m\geq i$, with $\oplus$ defined as before. Convexity of $F(c_1,c_2)$, that is $(\overline{Q}_4)$, yields that $\overline{x}^{n_0}_{m}\in F(c_1,c_2)$ for $m\geq n_0$. We now show that $d(\overline{x}_m,\overline{x}^{n_0}_{m})\to 0$ for $n_0\leq m\to\infty$. Concretely, for such an $m$, using joint convexity of the metric in Hadamard spaces (as the space has nonpositive curvature in the sense of Busemann), we have
\[
d(\overline{x}_{m+1},\overline{x}^{n_0}_{m+1})=d\left(\frac{\overline{t}_m}{\overline{t}_{m+1}}\overline{x}_m\oplus \frac{t_{m+1}}{\overline{t}_{m+1}}x_{m+1},\frac{\overline{t}_m}{\overline{t}_{m+1}}\overline{x}^{n_0}_m\oplus \frac{t_{m+1}}{\overline{t}_{m+1}}x_{m+1}\right)\leq \frac{\overline{t}_m}{\overline{t}_{m+1}}d(\overline{x}_m,\overline{x}^{n_0}_m)
\]
and so
\[
d(\overline{x}_{m},\overline{x}^{n_0}_{m})\leq \left(\prod_{i=n_0}^{m-1}\frac{\overline{t}_i}{\overline{t}_{i+1}}\right)d(\overline{x}_{n_0},\overline{x}^{n_0}_{n_0}) =\frac{\overline{t}_{n_0}}{\overline{t}_{m}}d(\overline{x}_{n_0},\overline{x}^{n_0}_{n_0}).
\]
As we have $\sum_{n\in\mathbb{N}}t_n=+\infty$, we get $d(\overline{x}_m,\overline{x}^{n_0}_{m})\to 0$ as $n_0\leq m\to\infty$. As $\overline{x}_{m_k}\to^w c_2$, we therefore get $\overline{x}^{n_0}_{m_k}\to^w c_2$. As $F(c_1,c_2)$ is convex and closed, it is weakly closed (see e.g.\ Lemma 3.2.1 in \cite{Bacak2014a}) and we hence have $c_2\in F(c_1,c_2)$. So in particular $d(c_2,c_1)\leq d(c_2,c_2)=0$, that is $c_1=c_2$.

Suppose now that $\lim_{n\to\infty}d^2(x_n,c_1)= \lim_{n\to\infty}d^2(x_n,c_2)=L$ and suppose for a contradiction that $c_1\neq c_2$. Let $\gamma$ be the unique geodesic joining $c_1$ and $c_2$, and write $\gamma'(\lambda):=\gamma(\lambda d(c_1,c_2))$. As $Z$ is convex, we get $\gamma'(\lambda)\in Z$ for all $\lambda\in [0,1]$. By assumption, we thus have that $d^2(x_n,\gamma'(\lambda))$ converges. The strong convexity of $d^2$ yields that
\[
d^2(x_n,\gamma'(\lambda))\leq (1-\lambda)d^2(x_n,c_1)+\lambda d^2(x_n,c_2)-\lambda(1-\lambda)d^2(c_1,c_2).
\]
Taking the limit as $n\to\infty$ thus yields
\[
\lim_{n\to\infty}d^2(x_n,\gamma'(\lambda))\leq L-\lambda(1-\lambda)d^2(c_1,c_2)<L
\]
for all $0<\lambda<1$, as $d^2(c_1,c_2)>0$. Therefore, one has $\lim_{n\to\infty}d^2(x_n,\gamma'(\lambda))< \lim_{n\to\infty}d^2(x_n,c_2)$ and hence can show as in the first part that $c_2\in F(\gamma'(\lambda),c_2)$. This gives
\[
d(c_2,\gamma'(\lambda))\leq d(c_2,c_2)=0
\]
for all $0<\lambda<1$, so that sending $\lambda$ to $0$ yields $c_1=c_2$, which is a contradiction.
\end{proof}

A direct lift to the stochastic is then the following proposition:

\begin{proposition}\label{weakErgodicConvergence}
Under the assumptions of Proposition \ref{weakConvergence}, suppose additionally that $X$ satisfies $(\overline{Q}_4)$, that $Z$ is convex and that $\mathfrak{W}(\overline{x}_n)\subseteq Z$ a.s.\ for a sequence of weighted averages $(\overline{x}_n)$ defined via some $(t_n)\subseteq (0,+\infty)$ with $\sum_{n\in\mathbb{N}}t_n=+\infty$.

Then $(\overline{x}_n)$ weakly converges a.s.\ to a $Z$-valued random variable.
\end{proposition}
\begin{proof}
The previous Proposition \ref{weakConvergence}, item (3), yields that there exists a set $\widetilde{\Omega}$ with measure one such that for all $\omega\in\widetilde{\Omega}$ and $z\in Z$, the sequence given by $d(x_n(\omega),z)$ converges. By assumption, we have $\mathfrak{W}(\overline{x}_n(\omega))\subseteq Z$ for all $\omega\in \widehat{\Omega}$, and $\widehat{\Omega}$ with measure one. In particular, we therefore have $\mathfrak{W}(\overline{x}_n(\omega))\subseteq Z$ and $d(x_n(\omega),z)$ converges for all $z\in Z$, for all $\omega\in \widehat{\Omega}\cap \widetilde{\Omega}$, and $\widehat{\Omega}\cap \widetilde{\Omega}$ still has measure one.

Lemma \ref{nonlinearPassty} applied for each $\omega\in \widehat{\Omega}\cap \widetilde{\Omega}$ individually yields that $(\overline{x}_n(\omega))$ weakly converges to some $x(\omega)\in Z$. As each $\overline{x}_n$ is measurable, Proposition \ref{nonlinearPettis} yields that there is a measurable $y$ with $x=y$ a.s. Thus $y$ is a.s.\ $Z$-valued and so w.l.o.g.\ we can assume that $y$ is $Z$-valued. Clearly, $(\overline{x}_n)$ converges weakly to $y$ almost surely.
\end{proof}

Let us also note that weak convergence implies weak ergodic convergence on Hadamard spaces satisfying $(\overline{Q}_4)$. For that, we rely on a nonlinear version of Brezis' tail lemma (see e.g.\ Exercise 3.13(1) in \cite{Brezis2011}) in such spaces as established by Kakavandi \cite{Kakavandi2013}:\footnote{Note that the easy inclusion $x\in \bigcap_{n\in\mathbb{N}}\overline{\mathrm{co}}\{x_k\mid k\geq n\}$ already appears as Proposition 3.2 in \cite{KirkPanyanak2008}.}

\begin{lemma}[Theorem 3.8 in \cite{Kakavandi2013}]\label{Brezis}
Let $X$ be a Hadamard space satisfying $(\overline{Q}_4)$ and let $(x_n)\subseteq X$ be such that $x_n\to^w x$. Then
\[
\bigcap_{n\in\mathbb{N}}\overline{\mathrm{co}}\{x_k\mid k\geq n\}=\{x\}.
\]
\end{lemma}

This allows us to prove the following lemma relating weak and weak ergodic convergence:

\begin{lemma}\label{weakToErgodic}
Let $X$ be a Hadamard space satisfying $(\overline{Q}_4)$ and let $(x_n)\subseteq X$ be such that $x_n\to^w x$. Then $\overline{x}_n\to^w x$ for a sequence of weighted averages $(\overline{x}_n)$ defined via any $(t_n)\subseteq (0,+\infty)$ with $\sum_{n\in\mathbb{N}}t_n=+\infty$.
\end{lemma}
\begin{proof}
As $(x_n)$ is bounded, so is $(\overline{x}_n)$. Hence, $(\overline{x}_n)$ has a weak cluster point (recall Proposition 3.1.2 in \cite{Bacak2014a}) and so it suffices to show that any such weak cluster point coincides with $x$. Hence, let $\overline{x}_{n_k}\to^w z$. Let $i\in\mathbb{N}$ be arbitrary and denote the weighted average of $(x_{n+i})$ with $(t_n)$ by $\overline{x}^{i}_{m}$ for $m\geq i$, similar as in the proof of Lemma \ref{nonlinearPassty}. Further similar to that proof, we get that $d(\overline{x}_m,\overline{x}^{i}_{m})\to 0$ for $i\leq m\to\infty$. We hence have $\overline{x}^i_{n_k}\to^w z$. Since $\overline{x}^i_{n_k}\in\mathrm{co}\{x_k\mid k\geq i\}$ for $k$ sufficiently large such that $n_k\geq i$, and since $\overline{\mathrm{co}}\{x_k\mid k\geq i\}$ is closed and convex and hence weakly closed (recall Lemma 3.2.1 in \cite{Bacak2014a}), we have $z\in \overline{\mathrm{co}}\{x_k\mid k\geq i\}$. Since $i\in\mathbb{N}$ was arbitrary, we have $z\in \bigcap_{i\in\mathbb{N}}\overline{\mathrm{co}}\{x_k\mid k\geq i\}$ and so Lemma \ref{Brezis} yields $z=x$. We hence have $\overline{x}_n\to^w x$.
\end{proof}

\section{Initial considerations on the method \eqref{SB}}\label{initial}

After these rather general considerations on the convergence of stochastic processes in Hadamard spaces, we now turn to the method \eqref{SB}. The present section is concerned with preliminary observations on that method which we derive before turning to our convergence results, in particular the measurability of the iteration \eqref{SB} under the assumptions (\textsf{A1}) -- (\textsf{A3}) and the existence of oracles satisfying (\textsf{A3}). In particular, from now on we assume, if not stated otherwise, that $X$ is a separable Hadamard space with the geodesic extension property and at least two points.

\subsection{Boundary cones and Busemann subgradients}

We first recall the basic definitions surrounding the method \eqref{SB}. We begin with the notions of boundary $X^\infty$ and boundary cone $CX^\infty$, where we refer to Chapter II.8 in \cite{BridsonHaefliger1999} for further discussions (see also \cite{GoodwinLewisLopezAcedoNicolae2025}).

Recall that the boundary of $X$ at infinity $X^\infty$ is the set of all equivalence classes of rays in $X$ under the equivalence relation of being asymptotic, where two rays $r,r'$ are called asymptotic if
\[
d(r(t),r'(t))\leq K\text{ for all }t\geq 0
\]
for some constant $K\geq 0$. As $X$ has the geodesic extension property and contains at least two points, $X^\infty$ is nonempty. We say that a ray $r$ has direction $\xi\in X^\infty$ if it belongs to the respective equivalence class. In particular, given any $x\in X$ and $\xi\in X^\infty$, there exists a unique ray issuing at $x$ with direction $\xi$. Following \cite{GoodwinLewisLopezAcedoNicolae2025}, we consider $X^\infty$ to be endowed with the cone topology (see Definition II.8.6 in \cite{BridsonHaefliger1999}). This makes the space first-countable, and moreover the topology is completely specified by its convergent sequences, with the key result in that vein discussed later in Lemma \ref{boundaryImpliesRay}. Indeed, we refer to \cite{GoodwinLewisLopezAcedoNicolae2025} for discussions on the particular benefits that the cone topology offers for optimisation purposes, compared to the perhaps more common topology induced by the angular metric.

The boundary cone $CX^\infty$ is now the usual Euclidean cone over $X^\infty$, i.e.\ $CX^\infty$ is the quotient of $X^\infty\times [0,\infty)$ under the equivalence relation 
\[
(\xi,s)\sim (\xi',s')\text{ if, and only if, }s=s'=0\text{ or }(\xi,s)=(\xi',s').
\]
Topologically, $X^\infty\times[0,\infty)$ is endowed with the product of the cone topology on $X^\infty$ and the usual metric topology on $[0,\infty)$, and $CX^\infty$ is endowed with the resulting quotient topology, also referred to as the cone topology. We denote an equivalence class of $(\xi,s)\in X^\infty\times [0,\infty)$ in $CX^\infty$ by $[\xi,s]$, and write $[0]$ for the equivalence class of $(\xi,0)$ for some/any $\xi\in X^\infty$. Again, convergence in $CX^\infty$ has a particularly useful characterization which we discuss later in Lemma \ref{boundaryConeimpliesBoundary}. Examples of $X^\infty$ and $CX^\infty$ for various spaces $X$ are discussed in \cite{GoodwinLewisLopezAcedoNicolae2025}.

We follow \cite{GoodwinLewisLopezAcedoNicolae2025} and define the ``pairing'' function $\langle\cdot,\cdot\rangle:X\times CX^\infty\to\mathbb{R}$ by
\[
\langle x,[\xi,s]\rangle:=\begin{cases}sb_{\xi}(x)&\text{if }s>0,\\0&\text{if }s=0,\end{cases}
\]
where we wrote $b_{\xi}$ for the Busemann function corresponding to the (unique) ray $r_{\overline{x},\xi}$ with direction $\xi$ and some arbitrary but fixed origin $\overline{x}$. Crucially, this pairing is continuous and $\langle \cdot,[\xi,s]\rangle$ is convex, $s$-Lipschitz and positively homogeneous in $s$, meaning that $\langle \cdot,[\xi,\alpha s]\rangle=\alpha\langle \cdot,[\xi,s]\rangle$ for all $[\xi,s]\in CX^\infty$ and $\alpha\geq 0$ (see Proposition 2.3 in \cite{GoodwinLewisLopezAcedoNicolae2025}).

We now turn to Busemann subgradients. Given a function $g:C\to\mathbb{R}$, a Busemann subgradient of $g$ at $x\in C$ is a point $[\xi,s]\in CX^\infty$ such that 
\[
x\in\mathrm{argmin}_{y\in C}\left(g(y)-\langle y,[\xi,s]\rangle\right).
\]
We call $g$ Busemann subdifferentiable if $g$ has a Busemann subgradient at every $x\in C$. As discussed in \cite{GoodwinLewisLopezAcedoNicolae2025}, $x$ minimizes $g$ if, and only if, $[0]$ is a Busemann subgradient of $g$ at $x$. The work \cite{GoodwinLewisLopezAcedoNicolae2025} provides quite a broad theory for such subgradients, including a chain rule, notions of conjugacy and geometry of level sets. We also refer to \cite{GoodwinLewisLopezAcedoNicolae2025} for various examples of Busemann subdifferentiable functions and further discussion on relation to subgradients in the sense of \cite{LewisLopezAcedoNicolae2024}, defined using tangent cones of Alexandrov spaces.

As for properties of Busemann subdifferentiable maps, we here mention the following: First, a Busemann subdifferentiable $g$ is lsc (see p.\ 11 in \cite{GoodwinLewisLopezAcedoNicolae2025}) and if $C$ is convex, then $g$ is also convex (see Proposition 3.3 in \cite{GoodwinLewisLopezAcedoNicolae2025}). Further, if $g$ is Busemann subdifferentiable such that each point $x\in C$ admits a subgradient $[\xi,s]$ with $s\leq L$, then $g$ is $L$-Lipschitz, with the converse being true if $C$ is open (see Proposition 4.3 in \cite{GoodwinLewisLopezAcedoNicolae2025}).

\subsection{Measurability of \eqref{SB}}

We now show the measurability of the main iteration under the assumptions \eqref{A1} -- \eqref{A3}. For that, we recall the following folklore result on Carath\'eodory functions.

\begin{lemma}[folklore, see e.g.\ Lemma 8.2.6 in \cite{AubinFrankowska2009}]\label{Cara}
Let $X,Y$ be complete separable metric spaces and let $(T,\mathcal{T})$ be a measurable space. If $g:T\times X\to Y$ is a Carath\'eodory function, then it is $\mathcal{T}\otimes\mathcal{B}(X)$/$\mathcal{B}(Y)$ measurable.
\end{lemma}

As a function $f:E\times C\to\mathbb{R}$ satisfying \eqref{A1} and \eqref{A2} is measurable in its left argument by \eqref{A1} and continuous in its right by \eqref{A2} (recall the above discussion by which $f(e,\cdot)$ is in particular $L(e)$-Lipschitz), it is thereby a Carath\'eodory function and so is $\mathcal{E}\otimes\mathcal{B}(C)$-measurable.

Further, we require the following two characterizations of convergence in $X^\infty$ as well as $CX^\infty$, both established in \cite{GoodwinLewisLopezAcedoNicolae2025} (see also the discussion therein for previous works mentioning this and related results).

\begin{lemma}[Proposition 2.2 in \cite{GoodwinLewisLopezAcedoNicolae2025}]\label{boundaryImpliesRay}
Fix $x\in X$. For any $(\xi_n)\subseteq X^\infty$ and $\xi\in X^\infty$, we have $\xi_n\to \xi$ if, and only if,
\[
r_{x,\xi_n}(\delta)\to r_{x,\xi}(\delta)
\]
for all $\delta>0$.
\end{lemma}

\begin{lemma}[Lemma 2.4 in \cite{GoodwinLewisLopezAcedoNicolae2025}]\label{boundaryConeimpliesBoundary}
For any $([\xi_n,s_n])\subseteq CX^\infty$ and $[\xi,s]\in CX^\infty$, if $[\xi_n,s_n]\to [\xi,s]$ in $CX^\infty$, then 
\[
\begin{cases}s_n\to s &\text{if }s=0,\\s_n\to s\text{ and }\xi_n\to \xi &\text{if }s>0.\end{cases}
\]
\end{lemma}

Immediately, we can now see that rays are continuous in both direction and origin.

\begin{lemma}\label{rayContinuous}
The ray $r_{x,\xi}(s)$, seen as a function $X\times CX^\infty\to X$, is continuous both in $x\in X$ and $[\xi,s]\in CX^\infty$.
\end{lemma}
\begin{proof}
As geodesic rays depend continuously on their origin, $r_{x,\xi}(s)$ is continuous in $x$. Now let $[\xi_n,s_n]\to [\xi,s]$ in $CX^\infty$. We have
\begin{align*}
d(r_{x,\xi_n}(s_n),r_{x,\xi}(s))&\leq d(r_{x,\xi_n}(s_n),r_{x,\xi_n}(s)) +d(r_{x,\xi_n}(s),r_{x,\xi}(s))\\
&=\vert s_n-s\vert+d(r_{x,\xi_n}(s),r_{x,\xi}(s))
\end{align*}
using the fact that the rays are geodesics. Lemma \ref{boundaryConeimpliesBoundary} now implies that either $s_n\to s$ if $s=0$, or $s_n\to s$ and $\xi_n\to \xi$ if $s>0$. In the former case, i.e.\ if $s=0$, we have $d(r_{x,\xi_n}(s),r_{x,\xi}(s))=d(x,x)=0$ so that $d(r_{x,\xi_n}(s_n),r_{x,\xi}(s))\leq \vert s_n-s\vert\to 0$. In the second case, i.e.\ if $s>0$, we get that $\xi_n\to \xi$ implies that $r_{x,\xi_n}(\delta)\to r_{x,\xi}(\delta)$ for all $\delta>0$ using Lemma \ref{boundaryImpliesRay}. In particular $d(r_{x,\xi_n}(s),r_{x,\xi}(s))\to 0$ as $s>0$. As $\vert s_n-s\vert\to 0$ as before, we get $d(r_{x,\xi_n}(s_n),r_{x,\xi}(s))\to 0$ all the same.
\end{proof}

We now establish measurability of the sequence \eqref{SB}.

\begin{lemma}\label{seqMeas}
Let $X$ be a separable Hadamard space with the geodesic extension property and at least two points, and let $C\subseteq X$ be closed, convex and nonempty. Let $f:E\times C\to \mathbb{R}$ be a given functional and assume \eqref{A1} -- \eqref{A3}.

Then the sequence $(x_n)$ defined by \eqref{SB} with $(t_n)$ and $(\zeta_{n+1})$, satisfying \eqref{parameters}, is measurable. In particular, $x_n$ is $\sigma(\zeta_1,\dots,\zeta_n)$-measurable.
\end{lemma}
\begin{proof}
By Lemma \ref{rayContinuous}, we have that $r_{x,\xi}(s)$ is continuous in $x\in X$ and $[\xi,s]\in CX^\infty$. In particular, it is $\mathcal{B}(CX^\infty)$/$\mathcal{B}(X)$ measurable in $[\xi,s]$. Thereby, $r_{x,\xi}(s)$ is a Carath\'eodory function and so is $\mathcal{B}(X)\otimes\mathcal{B}(CX^\infty)$/$\mathcal{B}(X)$ measurable by Lemma \ref{Cara} (taking $(CX^\infty,\mathcal{B}(CX^\infty))$ as $(T,\mathcal{T})$). As $P_C$ is continuous, it is measurable. Clearly $x_0\in X$ is measurable and assumption \eqref{A3} combined with the above yields by induction that
\[
x_{n+1}:=P_C(r_{x_n,\xi_{n}}(s_nt_n))
\]
is $\sigma(\zeta_{n+1},x_n)$-measurable for any $n\in\mathbb{N}$. In particular, we get that $x_n$ is $\sigma(\zeta_1,\dots,\zeta_n)$-measurable.
\end{proof}

\subsection{Measurable Busemann oracles}

It remains a question whether \eqref{A1} and \eqref{A2} suffice to guarantee the existence of an oracle satisfying \eqref{A3}, which we answer in this section, relying on a few further parts of measurable selection theory as well as on the topology of $CX^\infty$.

At first, we note that whenever $X$ is a separable Hadamard space, the space $X^\infty$ is actually a Polish space, that is it is metrizable so that the resulting metric space is complete and separable.

\begin{lemma}\label{XinftyPolish}
If $X$ is a separable Hadamard space, then $X^\infty$ is a Polish space.
\end{lemma}
\begin{proof}
We follow a standard strategy (see e.g.\ the proof of Proposition 1.1 in \cite{Attouch1979} as well as of Lemma 4.10 in the recent \cite{Pischke2026b}). Define the distance
\[
d_\infty(\xi,\xi'):=\sum_{n=0}^\infty \frac{1}{2^{n+1}}\min\{1,d(r_{x_n,\xi}(\delta_n),r_{x_n,\xi'}(\delta_n))\},
\]
where $\{(x_n,\delta_n)\mid n\in\mathbb{N}\}$ is dense in $X\times (0,\infty)$. It is easy to see that $d_\infty$ is a metric on $X^\infty$. Using Lemma \ref{boundaryImpliesRay}, and that geodesic rays are continuous and depend continuously on their origin, it is then straightforward to see that $d_\infty$ induces the cone topology on $X^\infty$. Hence $X^\infty$ is metrizable. To see that $(X^\infty,d_\infty)$ is separable, define the map $i:X^\infty\to X^\mathbb{N}$ via
\[
i(\xi):=(r_{x_n,\xi}(\delta_n))_{n\in\mathbb{N}},
\]
where $\{(x_n,\delta_n)\mid n\in\mathbb{N}\}$ is dense in $X\times (0,\infty)$ as before. The map $i$ is now injective: Suppose that $\xi,\xi'\in X^\infty$ are such that $r_{x_n,\xi}(\delta_n)=r_{x_n,\xi'}(\delta_n)$ for all $n\in\mathbb{N}$. Then, since geodesic rays are continuous and depend continuously on their origin, we get that $r_{x,\xi}(\delta)=r_{x,\xi'}(\delta)$ for all $\delta>0$ and $x\in X$. In particular, $r_{x,\xi}$ and $r_{x,\xi'}$ are asymptotic to each other for each $x\in X$, so that by definition of $X^\infty$ we have $\xi=\xi'$. In particular, $i$ is a bijection between $X^\infty$ and $i(X^\infty)$. Now, note that by definition of the cone topology, the map $i$ is a homeomorphism between $X^\infty$ and $i(X^\infty)$. As $X$ is separable, so is $X^\mathbb{N}$ and hence $X^\infty$ is also separable. It remains to see that $(X^\infty,d_\infty)$ is complete. For that, let $(\xi_k)$ be a Cauchy sequence in $X^\infty$, that is for any $n\in\mathbb{N}$:
\[
d(r_{x_n,\xi_k}(\delta_n),r_{x_n,\xi_{k'}}(\delta_n))\to 0
\]
as $k,k'\to \infty$. For that, first note that given $x,y\in X$ and $\xi\in X^\infty$, the function $\delta\mapsto d(r_{x,\xi}(\delta),r_{y,\xi}(\delta))$ is convex as $X$ is a Hadamard space and hence has nonpositive curvature in the sense of Busemann. Further, we have by definition of $X^\infty$ that $d(r_{x,\xi}(\delta),r_{y,\xi}(\delta))\leq K$ for all $\delta\geq 0$, for some constant $K\geq 0$. A bounded and convex function on $[0,\infty)$ must be non-increasing, so that we get
\[
d(r_{x,\xi}(\delta),r_{y,\xi}(\delta))\leq d(r_{x,\xi}(0),r_{y,\xi}(0))= d(x,y)\text{ for all }\delta\geq 0.
\]
Hence, using that a geodesic ray is an isometry, we get
\begin{align*}
d(r_{x,\xi_k}(\delta),r_{x,\xi_{k'}}(\delta))&\leq d(r_{x,\xi_k}(\delta),r_{x_n,\xi_{k}}(\delta_n))+d(r_{x_n,\xi_k}(\delta_n),r_{x_n,\xi_{k'}}(\delta_n))+d(r_{x_n,\xi_{k'}}(\delta_n),r_{x,\xi_{k'}}(\delta))\\
&\leq 2(d(x,x_n)+\vert \delta-\delta_n\vert)+d(r_{x_n,\xi_k}(\delta_n),r_{x_n,\xi_{k'}}(\delta_n))
\end{align*}
so that $d(r_{x,\xi_k}(\delta),r_{x,\xi_{k'}}(\delta))\to 0$ as $k,k'\to \infty$ for any $x\in X$ and $\delta>0$. As $X$ is complete, we have $r_{x,\xi_k}(\delta)\to y_{x,\delta}\in X$ as $k\to\infty$ for any $x\in X$ and $\delta\geq 0$. Being the pointwise limit of geodesic rays, we have that $\widehat{r}_x(\delta):=y_{x,\delta}$ is a geodesic ray (starting at $x$), for all $x\in X$. We now show that all $\widehat{r}_x$ are asymptotic to each other. As above, we have $d(r_{x,\xi_k}(\delta),r_{y,\xi_k}(\delta))\leq d(x,y)$ for all $\delta\geq 0$ so that
\begin{align*}
d(\widehat{r}_x(\delta),\widehat{r}_y(\delta))&\leq d(\widehat{r}_x(\delta),r_{x,\xi_k}(\delta))+d(r_{x,\xi_k}(\delta),r_{y,\xi_k}(\delta))+d(r_{y,\xi_k}(\delta),\widehat{r}_y(\delta))\\
&\leq d(\widehat{r}_x(\delta),r_{x,\xi_k}(\delta))+d(x,y)+d(r_{y,\xi_k}(\delta),\widehat{r}_y(\delta))\to d(x,y)
\end{align*}
as $k\to\infty$. Hence $d(\widehat{r}_x(\delta),\widehat{r}_y(\delta))\leq d(x,y)$, and so $\widehat{r}_x(\delta)$ and $\widehat{r}_y(\delta)$ are asymptotic. Define $\widehat{\xi}$ to be the equivalence class of some/any $\widehat{r}_x$ in $X^\infty$. It then follows that $r_{x,\widehat{\xi}}=\widehat{r}_x$ as there is a unique ray issuing at $x$ with direction $\widehat{\xi}$. In particular, Lemma \ref{boundaryImpliesRay} then yields that $\xi_k\to \widehat{\xi}$ in $X^\infty$.
\end{proof}

In particular, we can use this property of $X^\infty$ being a Polish space to prove that $CX^\infty$ is a Suslin space, that is $CX^\infty$ is the image of a Polish space under a continuous map.

\begin{lemma}\label{CXinftySuslin}
If $X$ is a separable Hadamard space, then $CX^\infty$ is a Suslin space.
\end{lemma}
\begin{proof}
Using Lemma \ref{XinftyPolish}, we get that $X^\infty$ and hence $X^\infty\times [0,\infty)$ are Polish spaces. The space $CX^\infty$ is now a continuous image of this space, namely under the canonical projection
\[
q:X^\infty\times [0,\infty)\to CX^\infty,\quad q(\xi,s):=[\xi,s].
\]
Indeed, note that $q$ is continuous exactly since $CX^\infty$ is endowed with the quotient topology.
\end{proof}

The space $CX^\infty$ being a Suslin space will later allow us to employ the following selection theorem to derive the existence of an oracle satisfying \eqref{A3}:

\begin{lemma}[see e.g.\ Theorem III.22 in \cite{CastaingValadier1977}]\label{AumannSelection}
Let $X$ be a Suslin space and let $(T,\mathcal{T},\tau)$ be a complete finite measure space. If $\varphi:T\to 2^X$ is a graph measurable map such that $\varphi(t)$ is nonempty for any $t\in T$, then there exists a measurable function $x: T\to X$ such that $x(t)\in\varphi(t)$ for all $t\in T$.
\end{lemma}

Over Lusin spaces, the above result is due to Aumann \cite{Aumann1969}. The extension to Suslin spaces is due to Sainte-Beuve \cite{SainteBeuve1974}.

Crucial for the rest of this section will be the Busemann subgradient, which we define via
\[
\partial_Bf:E\times C\to 2^{CX^\infty}, (e,x)\mapsto \{[\xi,s]\in CX^\infty\mid x\in \mathrm{argmin}_{y\in C}(f(e,y)-\langle y,[\xi,s]\rangle)\}.
\]
The existence of an oracle $\textsf{Busemann}_f$ satisfying assumption \eqref{A3} will in particular follow from a stronger result that shows that the composition of the map $\partial_Bf$ with functions $\zeta:\Omega\to E$ and $x:\Omega\to C$ is measurable in a suitable way.

First, we note that $\partial_Bf(e,\cdot)$ has a sequentially closed graph. For that, we rely on a characterization using Fenchel conjugacy. Following \cite{GoodwinLewisLopezAcedoNicolae2025} (see p.\ 16 therein), we define the Fenchel conjugate $f^\bullet:E\times CX^\infty\to (-\infty,+\infty]$ of $f$ by
\[
f^\bullet(e,[\xi,s]):=\sup_{y\in C}(\langle y,[\xi,s]\rangle-f(e,y)).
\]
Note that $f^\bullet(e,\cdot)$ is lsc as $\langle y,[\xi,s]\rangle$ is continuous in $[\xi,s]$ and that $f^\bullet(\cdot,[\xi,s])$ is measurable since $f(\cdot,y)$ is measurable for all $y\in C$, using also that $C$ is separable and that $\langle y,[\xi,s]\rangle-f(e,y)$ is continuous in $y\in C$.

Crucially, we have the following result on Fenchel duality:

\begin{lemma}[eq.\ (15) in \cite{GoodwinLewisLopezAcedoNicolae2025}]\label{FenchelDuality}
For any $e\in E$, $x\in C$ and $[\xi,s]\in CX^\infty$:
\[
f^\bullet(e,[\xi,s])+f(e,x)\geq\langle x,[\xi,s]\rangle,
\]
with equality if, and only if, $[\xi,s]\in\partial_B f(e,x)$.
\end{lemma}

We can now use this result to prove that $\partial_Bf(e,\cdot)$ has a sequentially closed graph (following essentially the usual argument for subdifferentials, see e.g.\ Proposition 16.36 in \cite{BauschkeCombettes2017}):

\begin{lemma}\label{clsdGraph}
Fix $e\in E$. Let $(x_n)\subseteq C$ and $([\xi_n,s_n])\subseteq CX^\infty$ be given with $[\xi_n,s_n]\in\partial_B f(e,x_n)$ as well as $x_n\to x$ and $[\xi_n,s_n]\to [\xi,s]$. Then $[\xi,s]\in\partial_Bf(e,x)$.
\end{lemma}
\begin{proof}
Recall that both $f$ and $f^\bullet$ are lsc in their right argument. By Lemma \ref{FenchelDuality}, we get $f^\bullet(e,[\xi_n,s_n])+f(e,x_n)=\langle x_n,[\xi_n,s_n]\rangle$ from $[\xi_n,s_n]\in\partial_B f(e,x_n)$. Hence, using Lemma \ref{FenchelDuality} again, we obtain
\begin{align*}
\langle x,[\xi,s]\rangle&\leq f^\bullet(e,[\xi,s])+f(e,x)\\
&\leq \liminf_{n\to\infty} f^\bullet(e,[\xi_n,s_n])+\liminf_{n\to\infty} f(e,x_n)\\
&\leq \liminf_{n\to\infty} (f^\bullet(e,[\xi_n,s_n])+f(e,x_n))\\
&=\lim_{n\to\infty}\langle x_n,[\xi_n,s_n]\rangle=\langle x,[\xi,s]\rangle,
\end{align*}
using in particular also again that $\langle\cdot,\cdot\rangle$ is continuous. We hence have $f^\bullet(e,[\xi,s])+f(e,x)=\langle x,[\xi,s]\rangle$ and Lemma \ref{FenchelDuality} yields $[\xi,s]\in\partial_Bf(e,x)$.
\end{proof}

We can also derive that the graph of $\partial_Bf$ is measurable.

\begin{lemma}\label{graphWeaklyMeas}
The map $\partial_Bf$ is graph measurable.
\end{lemma}
\begin{proof}
Define $\varphi(e,[\xi,s]):=\mathrm{argmin}_{y\in C}\left(f(e,y)-\langle y,[\xi,s]\rangle\right)$. Then for $g(e,x,[\xi,s]):=(e,[\xi,s],x)$, which is measurable, we have $\mathrm{gra}(\partial_Bf)=g^{-1}(\mathrm{gra}(\varphi))$. Note that $f(e,y)-\langle y,[\xi,s]\rangle$ is jointly measurable in $(e,[\xi,s])$ as well as continuous in $y$. Hence $\mathrm{gra}(\varphi)$ is measurable by Lemma \ref{minMeas} and so $\mathrm{gra}(\partial_Bf)$ is measurable.
\end{proof}

We now obtain the following result on the existence of almost sure Busemann subgradient oracles:

\begin{proposition}\label{oracleExistence}
Let $X$ be a separable Hadamard space with the geodesic extension property and at least two points. Let $C\subseteq X$ be closed, convex and nonempty. Further, let $f:E\times C\to \mathbb{R}$ be a given functional, assuming \eqref{A1} and \eqref{A2}. 

Then, for any measurable $x:\Omega\to C$ and $\zeta:\Omega\to E$, the map $\partial_Bf(\zeta,x):\Omega\to 2^{CX^\infty}$ is graph measurable. In particular, there is a $\sigma(\zeta,x)$-measurable function $\mathsf{Busemann}_f(\zeta,x)$ such that $\mathsf{Busemann}_f(\zeta,x)\in\partial_Bf(\zeta,x)$ almost surely.
\end{proposition}
\begin{proof}
Lemma \ref{graphWeaklyMeas} yields that $\partial_Bf$ is graph measurable. To see that $\partial_Bf(\zeta,x)$ is graph measurable, note that $\mathrm{gra}(\partial_Bf(\zeta,x))=g^{-1}(\mathrm{gra}(\partial_Bf))$ for $g(\omega,[\xi,s]):=(\zeta(\omega),x(\omega),[\xi,s])$. As $g$ is $\sigma(\zeta,x)\otimes \mathcal{B}(CX^\infty)$-measurable and $\partial_Bf$ is graph measurable, also $\partial_Bf(\zeta,x)$ is graph measurable, relative to $\sigma(\zeta,x)\otimes \mathcal{B}(CX^\infty)$ and hence also relative to $\overline{\sigma(\zeta,x)}\otimes \mathcal{B}(CX^\infty)$ where $(\Omega,\overline{\sigma(\zeta,x)},\overline{\PP})$ is the completion of $(\Omega,\sigma(\zeta,x),\PP)$. Lemma \ref{AumannSelection} now yields the existence of a $\overline{\sigma(\zeta,x)}$-measurable $\mathsf{Busemann}'_f(\zeta,x)$ with $\mathsf{Busemann}'_f(\zeta,x)\in\partial_Bf(\zeta,x)$ pointwise everywhere. Using Lemma \ref{removeCompleteness}, we obtain a $\sigma(\zeta,x)$-measurable function $\mathsf{Busemann}_f(\zeta,x)$ which is equal to $\mathsf{Busemann}'_f(\zeta,x)$ almost surely, so that $\mathsf{Busemann}_f(\zeta,x)\in\partial_Bf(\zeta,x)$ almost surely. For that, note in particular that $\mathcal{B}(CX^\infty)$ is countably generated since $CX^\infty$ is a Suslin space.
\end{proof}

The above Proposition \ref{oracleExistence} generally only guarantees the existence of measurable oracles which almost surely select Busemann subgradients. While sufficient for \eqref{A3} and hence for the main results of this paper, it remains a natural question on when a pointwise everywhere selection can be achieved. As we show in the following, such a result can be achieved if $X$ is locally compact.

For that, we require some further topology: if $X$ is a locally compact Hadamard space, then $CX^\infty$ is a locally compact Polish space, i.e.\ it is metrizable so that the resulting metric space is locally compact, complete and separable (see p.\ 7 of \cite{BertrandKloeckner2012}). In particular, $CX^\infty$ is $\sigma$-compact in this case, that is $CX^\infty$ can be represented as a countable union of compact subspaces.

Using this property of $CX^\infty$ now allows us to utilize the following composition lemma, which in works such as \cite{AubinFrankowska2009} (see Theorem 8.2.8 therein, which provides even more general results) is usually derived under a completeness assumption. We however here need to dispense of that, and for that follow the approach of Rockafellar \cite{Rockafellar1976b} (whose proof goes through for our special case, albeit being originally phrased for Euclidean spaces):

\begin{lemma}[essentially Theorem 1N in \cite{Rockafellar1976b}]\label{composition}
Let $X,Y$ be complete separable metric spaces, with $Y$ $\sigma$-compact, and $(T,\mathcal{T})$ be a measurable space. Let $A:T\times X\to 2^Y$ be such that $t\mapsto \mathrm{gra}(A(t,\cdot))$ is (weakly) measurable with closed values. Then for all measurable $x:T\to X$, the map
\[
\varphi(t):=A(t,x(t))
\]
is (weakly) measurable.
\end{lemma}
\begin{proof}
Let $C\subseteq Y$ be open. Using that $Y$ is $\sigma$-compact, we express $C$ as the union over a sequence of compact sets $(C_k)$. Define the map $\varphi_k(t):=\{x(t)\}\times C_k$, and note that $\varphi_k$ has compact images for every $k$. By Lemma 18.4 in \cite{AliprantisBorder2006}, each $\varphi_k$ is (weakly) measurable. As $C$ is open and $\varphi(t)$ is closed, we have
\begin{align*}
\varphi^{-1}(C)&=\{t\in T\mid C\cap A(t,x(t))\neq\emptyset\}\\
&=\bigcup_{k\in\mathbb{N}}\{t\in T\mid (x,y)\in \mathrm{gra}(A(t,\cdot))\text{ and }x=x(t),y\in C_k\}\\
&=\bigcup_{k\in\mathbb{N}}\{t\in T\mid \mathrm{gra}(A(t,\cdot))\cap\varphi_k(t)\neq\emptyset\}.
\end{align*}
By Lemma 18.4 in \cite{AliprantisBorder2006}, we have that $\mathrm{gra}(A(t,\cdot))\cap\varphi_k(t)$ is (weakly) measurable (where it is crucial that $\varphi_k$ has compact images). In particular, the set 
\[
\{t\in T\mid \mathrm{gra}(A(t,\cdot))\cap\varphi_k(t)\neq\emptyset\}=(\mathrm{gra}(A(t,\cdot))\cap\varphi_k(t))^{-1}(X\times Y)
\]
and hence $\varphi^{-1}(C)$ is measurable. As $C$ was arbitrary, we get that $\varphi$ is (weakly) measurable.
\end{proof}

We now obtain the following result on the existence of pointwise everywhere Busemann subgradient oracles, making again use of the Kuratowski–Ryll-Nardzewski selection theorem (recall Lemma \ref{Selection}):

\begin{proposition}\label{oracleExistenceProper}
Let $X$ be a separable Hadamard space with the geodesic extension property and at least two points, and assume that $X$ is locally compact. Let $C\subseteq X$ be closed, convex and nonempty. Further, let $f:E\times C\to \mathbb{R}$ be a given functional, assuming \eqref{A1} and \eqref{A2}. 

Then, for any measurable $x:\Omega\to C$ and $\zeta:\Omega\to E$, the map $\partial_Bf(\zeta,x):\Omega\to 2^{CX^\infty}$ is (weakly) measurable. In particular, there is a $\sigma(\zeta,x)$-measurable function $\mathsf{Busemann}_f(\zeta,x)$ such that $\mathsf{Busemann}_f(\zeta,x)\in\partial_Bf(\zeta,x)$ pointwise everywhere.
\end{proposition}
\begin{proof}
Lemma \ref{graphWeaklyMeas} yields that $\partial_B f$ is graph measurable. Using the completeness of $(E,\mathcal{E},\mu)$ and that the graph of $\partial_Bf(e,\cdot)$ is closed as shown in Lemma \ref{clsdGraph}, it follows that $e\mapsto \mathrm{gra}(\partial_Bf(e,\cdot))$ is (weakly) measurable (recall Theorem 8.1.4 in \cite{AubinFrankowska2009}). In particular, the map $\omega\to \mathrm{gra}(\partial_Bf(\zeta(\omega),\cdot))$ is (weakly) measurable relative to $\sigma(\zeta,x)$, so that (weak) measurability of $\partial_Bf(\zeta,x)$ relative to $\sigma(\zeta,x)$ follows using Lemma \ref{composition} since $X$ is locally compact and hence $CX^\infty$ is a locally compact Polish space, and so $CX^\infty$ is $\sigma$-compact. As $\partial_Bf(\zeta,x)$ is closed, Lemma \ref{Selection} yields the existence of a $\sigma(\zeta,x)$-measurable selection $\mathsf{Busemann}_f(\zeta,x)$ with the above properties.
\end{proof}

The last result regarding Busemann subgradient oracles that we discuss is concerned with the special case where $(E,\mathcal{E},\mu)$ is a discrete space, that is where $E$ is countable and $\mathcal{E}=2^E$ (slightly extending the setting of \cite{GoodwinLewisLopezAcedoNicolae2025} where only finite discrete spaces $(E,\mathcal{E},\mu)$ are considered). In that case, the above results can be strengthened in a somewhat different way, namely in the sense that any function $\mathsf{Busemann}_f:E\times C\to CX^\infty$ with $\mathsf{Busemann}_f(e,x)\in\partial_Bf(e,x)$ for all $e\in E$ and $x\in C$ induces an oracle that satisfies \eqref{A3}. 

\begin{proposition}\label{oracleDiscrete}
Let $X$ be a separable Hadamard space with the geodesic extension property and at least two points. Let $C\subseteq X$ be closed, convex and nonempty. Further, let $f:E\times C\to \mathbb{R}$ be a given functional, assuming \eqref{A1} and \eqref{A2}. Let $(E,\mathcal{E},\mu)$ be discrete.

Then any function $\mathsf{Busemann}_f:E\times C\to CX^\infty$ such that $\mathsf{Busemann}_f(e,x)\in\partial_Bf(e,x)$ for all $e\in E$ and $x\in C$ induces an oracle that satisfies \eqref{A3}, that is $\mathsf{Busemann}_f(\zeta_{n+1},x_n):\Omega\to CX^\infty$ is $\sigma(\zeta_{n+1},x_n)$-measurable for any $n\in\mathbb{N}$.
\end{proposition}
\begin{proof}
We prove the claim by induction on $n\in\mathbb{N}$, simultaneously establishing that $x_n$ is $\sigma(\zeta_1,\dots,\zeta_{n})$-measurable and only assumes countably many values. For $n=0$, we clearly have the latter as $x_0\in C$ is a single point. Now, for $A\in\mathcal{B}(CX^\infty)$, we get that
\[
\mathsf{Busemann}_f(\zeta_{1},x_0)^{-1}(A)=\bigcup_{e\in E, \mathsf{Busemann}_f(e,x_0)\in A} \zeta_1^{-1}(\{e\})\in\sigma(\zeta_{1}),
\]
using that $E$ is countable. The same representation now deals with the induction step: Using that $x_n$ is $\sigma(\zeta_{1},\dots,\zeta_n)$-measurable and only assumes countably many values as assumed in the induction hypothesis, we have
\[
\mathsf{Busemann}_f(\zeta_{n+1},x_n)^{-1}(A)=\bigcup_{e\in E, x\in\mathrm{img}(x_n), \mathsf{Busemann}_f(e,x)\in A} \zeta_{n+1}^{-1}(\{e\})\cap x_n^{-1}(\{x\})\in\sigma(\zeta_{n+1},x_n),
\]
using that $E$ as well as $\mathrm{img}(x_n)$ are countable. That $x_{n+1}$ is $\sigma(\zeta_{1},\dots,\zeta_{n+1})$-measurable now follows as in the proof of Lemma \ref{seqMeas} and that $x_{n+1}$ only assumes countably many values follows since $x_n$ and $\zeta_{n+1}$ only assume countably many values, so that also $\mathsf{Busemann}_f(\zeta_{n+1},x_n)$ and hence $x_{n+1}=P_C(r_{x_n,\xi_{n}}(s_nt_n))$ do.
\end{proof}

\section{Proofs of the main results}

We now turn to our convergence results for \eqref{SB}. For that, we now briefly recall the setup: Let $(E,\mathcal{E},\mu)$ and $(\Omega,\mathcal{F},\PP)$ be probability spaces, with $(E,\mathcal{E},\mu)$ complete, exactly as before, and let $X$ be a separable Hadamard space with the geodesic extension property and at least two points. Further, fix a closed, convex and nonempty subset $C\subseteq X$ and let $f:E\times C\to \mathbb{R}$ be a function with properties \eqref{A1} -- \eqref{A3} as above, and write $\underline{L}:=\int L^2\,d\mu$. Let $(x_n)$ be the iteration given by \eqref{SB}, and assume \eqref{parameters}. Write $F(x):=\int f(e,x)\,d\mu(e)$ and assume that $\mathrm{argmin}F\neq\emptyset$. Note that since each $f(e,\cdot)$ is $L(e)$-Lipschitz by assumption \eqref{A2} and since $\vert F(x)\vert <+\infty$ for some $x\in C$ as $\mathrm{argmin}F\neq\emptyset$, the function $F$ is well-defined for all $x\in C$ and Lipschitz continuous with constant $\int L\,d\mu$.

All the convergence proofs now proceed by utilizing that the iteration in question is stochastically quasi-Fej\'er monotone. To first show that, we introduce some notation. Define the filtration $\mathsf{F}_n:=\sigma(\zeta_1,\dots,\zeta_n)$ and write $\EE_n$ for the conditional expectation $\mathbb{E}[\cdot \mid\mathsf{F}_n]$.

The key geometric ingredient is the following property of the Busemann subgradients:

\begin{lemma}[see e.g.\ Lemma 6.1 in \cite{GoodwinLewisLopezAcedoNicolae2025}]\label{BSLemma}
Let $g:C\to \mathbb{R}$ be a given function for a nonempty closed and convex set $C\subseteq X$ and let $[\xi,s]$ be a Busemann subgradient of $g$ at $x\in C$. Given $t>0$, define
\[
x^+:=\begin{cases}P_C(r_{x,\xi}(st))&\text{if }s>0,\\x&\text{if }s=0.\end{cases}
\]
Then for any $y\in C$:
\[
d^2(x^+,y)\leq d^2(x,y)-2t(g(x)-g(y))+s^2t^2.
\]
\end{lemma}

This immediately allows us to derive the following stochastic variant of quasi-Fej\'er monotonicity.

\begin{lemma}\label{IneqMain}
For any $n\in\mathbb{N}$ and any $y\in C$:
\[
\EE_n[d^2(x_{n+1},y)]\leq d^2(x_n,y)-2t_n(F(x_n)-F(y))+\underline{L}t_n^2.
\]
\end{lemma}
\begin{proof}
Fix $y\in C$ and $n\in\mathbb{N}$. Using Lemma \ref{BSLemma}, we get
\[
d^2(x_{n+1},y)\leq d^2(x_n,y)-2t_n(f(\zeta_{n+1},x_n)-f(\zeta_{n+1},y))+s_n^2t_n^2.
\]
By assumption \eqref{A2}, we get $s_n\leq L(\zeta_{n+1})$ and applying conditional expectations yields
\[
\EE_n[d^2(x_{n+1},y)]\leq d^2(x_n,y)-2t_n(\EE_n[f(\zeta_{n+1},x_n)]-\EE_n[f(\zeta_{n+1},y)])+\underline{L}t_n^2,
\]
using that $\zeta_{n+1}$ has distribution $\mu$. Now, as $\zeta_{n+1}$ is independent of $x_n$ and $\mathsf{F}_n$, using \eqref{parameters} and that $x_n$ is $\mathsf{F}_n$-measurable by Lemma \ref{seqMeas}, we have 
\[
\EE_n[f(\zeta_{n+1},x_n)](\omega)=\int f(\zeta_{n+1}(\omega'),x_n(\omega))\,d\PP(\omega')=F(x_n(\omega))
\]
and similarly $\EE_n[f(\zeta_{n+1},y)]=F(y)$. This yields the claim.
\end{proof}

We can now prove Theorem \ref{SBstrongConv}, that is the almost sure convergence of \eqref{SB} under a local compactness assumption:

\begin{proof}[Proof of Theorem \ref{SBstrongConv}]
By Lemma \ref{IneqMain}, we have 
\[
\EE_n[d^2(x_{n+1},z)]\leq d^2(x_n,z)-2t_n(F(x_n)-\mathrm{min}F)+\underline{L}t_n^2
\]
for any $z\in\mathrm{argmin}F$, for which we fix one in the following (using $\mathrm{argmin}F\neq\emptyset$). Therefore, using \eqref{parameters}, we find that \eqref{randFejer} in Proposition \ref{weakConvergence} is satisfied. Items (1) and (2) of that result now yield that $(x_n)$ is bounded and $\sum_{n=0}^\infty t_n(F(x_n)-\mathrm{min}F)<+\infty$ a.s., say jointly on a set $\widehat{\Omega}$ of measure one. Fix one such $\omega\in\widehat{\Omega}$. Using \eqref{parameters} again, this in particular implies 
\[
\liminf_{n\to\infty} F(x_n(\omega))=\mathrm{min}F.
\]
Let $F(x_{n_k}(\omega))\to \mathrm{min} F$. Using local compactness, choose a convergent subsequence $x_{n_{k_j}}(\omega)\to x(\omega)$. In particular, using that $F$ is continuous, we have
\[
F(x(\omega))=\lim_{j\to\infty} F(x_{n_{k_j}}(\omega))=\mathrm{min} F
\]
so that $x(\omega)\in\mathrm{argmin}F$ for any $\omega\in \widehat{\Omega}$. In particular, we thus have that $\mathfrak{S}(x_n)\cap\mathrm{argmin}F\neq\emptyset$ a.s., so that item (5) of Proposition \ref{weakConvergence} yields that $(x_n)$ strongly converges a.s.\ to an $\mathrm{argmin}F$-valued random variable.
\end{proof}

Further, we can now also provide a proof of Theorem \ref{SBweakErgodicConv}, that is for the weak almost sure convergence of \eqref{SB}. For that, as already discussed in the introduction, we rely on the following result established in \cite{Pischke2026}:

\begin{lemma}[Lemma 3.4 in \cite{Pischke2026}]\label{LipschitzSum}
Let $(\lambda_n)\subseteq (0,\infty)$ with $\sum_{n\in\mathbb{N}}\lambda_n=+\infty$ and $\sum_{n\in\mathbb{N}}\lambda_n^2<+\infty$. Assume $(\alpha_n)$ is a sequence of nonnegative square-integrable random variables which is i.i.d. Further, assume that $(\gamma_n)$ satisfies $\sup_{n\in\mathbb{N}} \gamma_n<+\infty$ a.s. Let $(\beta_n)$ be a sequence of nonnegative random variables such that $\sum_{n\in\mathbb{N}}\lambda_n\beta_n<+\infty$ a.s.\ and
\[
\beta_{n+1}-\beta_n \leq \theta \lambda_{n}\gamma_n\alpha_{n}
\]
a.s.\ for all $n\in\mathbb{N}$, and some constant $\theta>0$. Then $\beta_n\to 0$ a.s.
\end{lemma}

Using that lemma, we can give the following proof of Theorem \ref{SBweakErgodicConv}, proceeding similar to \cite{Pischke2026}:

\begin{proof}[Proof of Theorem \ref{SBweakErgodicConv}]
As in the proof of Theorem \ref{SBstrongConv}, we derive
\[
\EE_n[d^2(x_{n+1},z)]\leq d^2(x_n,z)-2t_n(F(x_n)-\mathrm{min}F)+\underline{L}t_n^2
\]
for any $z\in\mathrm{argmin}F$ from Lemma \ref{IneqMain}. Using \eqref{parameters}, the assumptions of Proposition \ref{weakConvergence} are satisfied. Items (1) and (2) of that result imply that $(x_n)$ is bounded and
\[
\sum_{n\in\mathbb{N}} t_n(F(x_n)-\mathrm{min}F)<+\infty\text{ a.s.}
\]
As $F$ is Lipschitz continuous with constant $\int L\,d\mu$, i.e.\ $\vert F(x)-F(y)\vert\leq \left(\int L\,d\mu\right)d(x,y)$ for all $x,y\in C$, we get for $\beta_n:=F(x_n)-\min F$ that
\[
\beta_{n+1}-\beta_n= F(x_{n+1})-F(x_n)\leq \left(\int L\,d\mu\right)d(x_{n+1},x_n).
\]
Using Lemma \ref{BSLemma}, we obtain $d^2(x_{n+1},x_n)\leq s_n^2t_n^2\leq L^2(\zeta_{n+1})t_n^2$ so that we can derive
\[
\beta_{n+1}-\beta_n\leq \left(\int L\,d\mu\right)t_{n}L(\zeta_{n+1})
\]
for all $n\in\mathbb{N}$. So, also using \eqref{parameters}, the assumptions of Lemma \ref{LipschitzSum} are met with $\beta_n$ as above, as well as $\lambda_n:=t_n$, $\gamma_n:=1$, $\theta:=\int L\,d\mu$ and $\alpha_n:=L(\zeta_{n+1})$, so that we obtain $\beta_n\to 0$ a.s., that is
\[
F(x_n)\to\min F\text{ a.s.},
\]
say on a set $\widehat{\Omega}$ with measure one. On that set $\widehat{\Omega}$, we now also have $\mathfrak{W}(x_n)\subseteq \mathrm{argmin}F$. To see that, let $\omega\in \widehat{\Omega}$ and let $x(\omega)$ be a weak cluster point of $(x_n(\omega))$, i.e.\ $x_{n_k}(\omega)\to^w x(\omega)$ for some subsequence $(x_{n_k}(\omega))$. As $F$ is convex and Lipschitz continuous, it is hence weakly lsc (recall Lemma 3.1 in \cite{Bacak2013}), and since $\omega\in\widehat{\Omega}$, we have
\[
F(x(\omega))\leq \liminf_{k\to\infty} F(x_{n_k}(\omega))=\lim_{n\to\infty} F(x_n(\omega))=\mathrm{min}F
\]
and so $x(\omega)\in\mathrm{argmin}F$. Therefore, we have $\mathfrak{W}(x_n)\subseteq \mathrm{argmin}F$ a.s., and so item (4) of Proposition \ref{weakConvergence} implies that $(x_n)$ a.s.\ weakly converges to an $\mathrm{argmin}F$-valued random variable.
\end{proof}

At last, we turn to the quantitative results given in Theorems \ref{newResidual} and \ref{SBrates}. The former is rather immediate from Lemma \ref{IneqMain}, and proceeds very similar to the proof of Theorem 6.2 in \cite{GoodwinLewisLopezAcedoNicolae2025} (and of Theorem 8.35(a) in \cite{Beck2017}):

\begin{proof}[Proof of Theorem \ref{newResidual}]
Using Lemma \ref{IneqMain}, we have
\[
\EE_n[d^2(x_{n+1},z)]\leq d^2(x_n,z)-2t_n(F(x_n)-\min F)+\underline{L}t_n^2
\]
for any $n\in\mathbb{N}$, where $z\in\mathrm{argmin}F$. Taking expectations, we get
\[
\EE[d^2(x_{n+1},z)]\leq \EE[d^2(x_n,z)]-2t_n\EE[F(x_n)-\min F]+\underline{L}t_n^2.
\]
Rearranging and summing up yields that for any $k\in\mathbb{N}$ and any $n\geq k$:
\[
2\sum_{i=k}^n t_i\EE[F(x_i)-\min F]\leq \EE[d^2(x_k,z)]+\underline{L}\sum_{i=k}^n t_i^2.
\]
Using the inequality $\EE[\min_{i=k,\dots,n} F(x_i)]\leq  \min_{i=k,\dots,n}\EE[F(x_i)]$ (see e.g.\ the proof of Theorem 8.35(a) in \cite{Beck2017}), we obtain the estimate
\begin{align*}
\EE[\min_{i=k,\dots,n} F(x_i)-\min F]&\leq \min_{i=k,\dots,n}\EE[ F(x_i)-\min F]\\
&\leq \frac{2 \sum_{i=k}^n t_i\EE[F(x_i)-\min F]}{2 \sum_{i=k}^n t_i}\\
&\leq \frac{\EE[d^2(x_k,z)]+\underline{L}\sum_{i=k}^n t_i^2}{2 \sum_{i=k}^n t_i}.
\end{align*}
If now $\sum_{k=0}^nt_k^2/\sum_{k=0}^nt_k\to 0$, taking $k=0$ above then yields $\EE[\min_{i=0,\dots,n} F(x_i)-\min F]\to 0$ as claimed. If further $C$ has diameter bounded by $D>0$, one has
\[
\EE[\min_{i=0,\dots,n} F(x_i)-\min F]\leq\EE[\min_{i=k,\dots,n} F(x_i)-\min F]\leq \frac{D^2+\underline{L}\sum_{i=k}^n t_i^2}{2 \sum_{i=k}^n t_i}.
\]
Taking $k=\lceil n/2\rceil$ as well as $t_n:=\frac{D}{\sqrt{\underline{L}}\sqrt{n+1}}$ in the above and using Lemma 8.27(b) in \cite{Beck2017} then yields
\[
\EE\left[\min_{i=0,\dots,n}F(x_i) -\min F\right]\leq \frac{2(1+\log(3))D\sqrt{\underline{L}}}{\sqrt{n+2}}
\]
for all $n\geq 2$, as claimed.
\end{proof}

For the proof of Theorem \ref{SBrates}, we require the following quantitative version of a lemma of Qihou \cite{Qihou2001} (see also Lemma 5.31 in \cite{BauschkeCombettes2017}):

\begin{lemma}[Theorem 3.2 in \cite{NeriPowell2024}]\label{qihou}
Let $(x_n)$, $(\alpha_n)$, $(\beta_n)$ and $(\gamma_n)$ be sequences of nonnegative reals with
\[
x_{n+1}\leq (1+\alpha_n)x_n-\beta_n+\gamma_n
\]
for all $n\in\mathbb{N}$. If $\prod_{i=0}^\infty(1+\alpha_i)<\infty$ and $\sum_{i=0}^\infty\gamma_i<\infty$, then $(x_n)$ converges and $\sum_{i=0}^\infty\beta_i<\infty$.

Further, if $K,L,M>0$ satisfy $x_0<K$, $\prod_{i=0}^\infty(1+\alpha_i)<L$ and $\sum_{i=0}^\infty\gamma_i<M$, then $\sum_{i=0}^\infty\beta_i<L(K+M)$.
\end{lemma}

The next result on approximation properties for summable sequences is folklore:

\begin{lemma}\label{sumconv}
Suppose that $(u_n)$, $(v_n)$ are sequences of nonnegative reals with $L>0$ such that $\sum_{n=0}^\infty u_nv_n< L$ and $\theta:\mathbb{N}\times(0,\infty)\to \mathbb{N}$ such that $\sum_{n=k}^{\theta(k,b)} u_n\geq b$ for all $b>0$ and $k\in\mathbb{N}$. Then $\liminf_{n\to\infty}v_n=0$ with 
\[
\forall\varepsilon>0\ \forall N\in\mathbb{N}\ \exists n\in [N;\theta(N,L/\varepsilon)](v_n<\varepsilon).
\]
\end{lemma}
\begin{proof}
For arbitrary $\varepsilon>0$ and $N\in\mathbb{N}$, suppose for a contradiction that $v_n\geq\varepsilon$ for all $n\in [N;\theta(N,L/\varepsilon)]$. Then $L\leq\varepsilon\sum_{n=N}^{\theta(N,L/\varepsilon)}u_n\leq \sum_{n=N}^{\theta(N,L/\varepsilon)} u_nv_n\leq \sum_{n=0}^{\infty} u_nv_n < L$, which is a contradiction.
\end{proof}

With these in place, we now derive a first asymptotic approximation result, extending Theorem \ref{old}:

\begin{lemma}\label{liminf}
Let $\theta:\mathbb{N}\times (0,\infty)\to\mathbb{N}$ be such that $\sum_{n=k}^{\theta(k,b)}t_n\geq b$ for all $b>0$ and $k\in\mathbb{N}$. Further, let $T > \sum_{n=0}^\infty t_n^2$. Lastly, let $b>0$ be such that $b>d^2(x_0,x^*)$ for some minimizer $x^*$ of $F$. Then $\liminf_{n\to\infty}\EE[F(x_n)]=\min F$ with
\[
\forall\varepsilon>0\ \forall N\in\mathbb{N}\ \exists n\in [N;\theta(N,(b+\underline{L}T)/\varepsilon)](\EE[F(x_n)]-\min F<\varepsilon).
\]
\end{lemma}
\begin{proof}
We have
\[
\EE_n[d^2(x_{n+1},x^*)]\leq d^2(x_n,x^*)-2t_n(F(x_n)-\min F)+\underline{L}t_n^2
\]
by Lemma \ref{IneqMain}. Therefore, Lemma \ref{qihou} yields $\sum_{n=0}^\infty t_n\EE[F(x_n)-\min F]<b+\underline{L}T$ and so Lemma \ref{sumconv} yields the result.
\end{proof}

We can now prove Theorem \ref{SBrates}:

\begin{proof}[Proof of Theorem \ref{SBrates}]
For any $n\in\mathbb{N}$, define $X_{n}:=d^2(x_n,x^*)+\underline{L}\sum^\infty_{m=n}t_m^2$. As $(x_n)$ is adapted to $(\mathcal{F}_n)$, also $(X_{n})$ is adapted to $(\mathcal{F}_n)$. As we have
\[
\EE_n[d^2(x_{n+1},x^*)]\leq d^2(x_n,x^*)-2t_n(F(x_n)-\min F)+\underline{L}t_n^2\leq d^2(x_n,x^*)+\underline{L}t_n^2
\]
by Lemma \ref{IneqMain}, the stochastic process $(X_{n})$ is a nonnegative supermartingale. Indeed, note that 
\begin{align*}
\EE_n[X_{n+1}]&=\EE_n\left[d^2(x_{n+1},x^*)\right]+\underline{L}\sum^\infty_{m={n+1}}t_m^2\leq d^2(x_n,x^*) + \underline{L}\sum^\infty_{m=n}t_m^2=X_{n}.
\end{align*}
Using the fact that $f(e,\cdot)$ is strongly convex, we get that $F$ is strongly convex with parameter $\underline{\alpha}$, i.e.
\[
F(\gamma(tl))\leq (1-t)F(\gamma(0))+tF(\gamma(l)) -t(1-t)\frac{\underline{\alpha}}{2}d^2(\gamma(0),\gamma(l))
\]
for any geodesic $\gamma:[0,l]\to C$ and any $t\in [0,1]$. For $\gamma$ being the unique geodesic joining $x_n$ and $x^*$, and $l=d(x_n,x^*)$, we get
\[
\min F\leq F\left(\gamma\left(\frac{l}{2}\right)\right)\leq \frac{1}{2}F(x_n)+\frac{1}{2}\min F -\frac{\underline{\alpha}}{8}d^2(x_n,x^*)
\]
so that
\[
\frac{\underline{\alpha}}{4}d^2(x_n,x^*)\leq F(x_n)-\min F.
\]
Now, let $\varepsilon>0$ be arbitrary and using Lemma \ref{liminf}, choose an
\[
n\in \left[\chi(\varepsilon/2\underline{L});\theta(\chi(\varepsilon/2\underline{L}),8(b+\underline{L}T)/\varepsilon\underline{\alpha})\right]
\]
such that $\EE[F(x_n)]-\min F<\varepsilon\underline{\alpha}/8$. Then $\EE[d^2(x_{n},x^*)]<\varepsilon/2$. Let $m\geq n$ be arbitrary. Then
\[
\EE[d^2(x_{m},x^*)]\leq \EE[X_m]\leq \EE[X_n]= \EE[d^2(x_{n},x^*)]+\underline{L}\sum^\infty_{m=n}t_m^2< \varepsilon
\]
using that $(X_m)$ is a supermartingale and the properties of $\chi$. As $m$ was arbitrary, this yields $\EE[d^2(x_{n},x^*)]\to 0$ and that $\rho$ is a rate of convergence for that limit.

For $d^2(x_{n},x^*)\to 0$ a.s., note that
\begin{align*}
\PP(\exists m\geq n(d^2(x_{m},x^*)\geq a))\leq \PP(\exists m\geq n(X_m\geq a))\leq \frac{\EE[X_n]}{a}
\end{align*}
where the second inequality follows from Ville's inequality \cite{Ville1939} (see also \cite{Metivier1982}). This immediately implies that $d^2(x_{n},x^*)\to 0$ a.s.\ with rate $\rho'$.
\end{proof}

\section{Applications to statistical models for phylogenetic inference}

To illustrate the practical use of \eqref{SB} more concretely, we discuss a particular application to phylogenetic inference in more detail, inspired by the work of Benner, Ba\v{c}\'ak, and Bourguignon \cite{BennerBacakBourguignon2014} (see also \cite{Bacak2018}). Write $\mathcal{T}_l$ for the so-called Billera-Holmes-Vogtmann tree space of phylogenetic trees on $l$ labelled leaves which was introduced in the seminal work \cite{BilleraHolmesVogtmann2001} and in particular shown therein to be a $\CAT$ cubical complex, and hence in particular a locally compact Hadamard space. Further, $\mathcal{T}_l$ has the geodesic extension property (which follows from Proposition II.5.10 in \cite{BridsonHaefliger1999} and the results of \cite{BilleraHolmesVogtmann2001}, see in particular p.\ 743 in the latter).

The work of Benner, Ba\v{c}\'ak, and Bourguignon \cite{BennerBacakBourguignon2014} introduced a statistical model for phylogenetic inference based on minimizing the function
\[
F(x):=\int_{\mathcal{T}_l} d(e,x)\,d\mu_D(e)
\]
over $x\in\mathcal{T}_l$, where $\mu_D$ represents a so-called posterior distribution on $\mathcal{T}_l$ (given data $D$, see in particular \cite{BennerBacakBourguignon2014} for further discussions). Following the notation of this paper, we are hence concerned with the function $f:E\times X\to\mathbb{R}$ given by $f(e,x):=d(e,x)$ where $E:=X:=C:=\mathcal{T}_l$ and $\mu:=\mu_D$. In particular, the function $f$ is Busemann subdifferentiable for any $e\in\mathcal{T}_l$, with any function
\[
\mathsf{Busemann}_f(e,x):=\begin{cases}[0]&\text{if }x=e,\\ [r,1]&\text{if }x\neq e,\end{cases}
\]
being a selection of the Busemann subdifferential $\partial_Bf$, where $r$ is a ray issuing from $x$ and passing through $e$ (see Example 3.1 in \cite{GoodwinLewisLopezAcedoNicolae2025}).

Clearly, $f(\cdot,x)$ is measurable for all $x\in\mathcal{T}_l$, and so the above $f$ satisfies assumption \eqref{A1}. Also \eqref{A2} follows immediately by setting $L(e):=1$ as $f(e,\cdot)$ is Lipschitz with constant $1$, so that every Busemann subgradient $[\xi,s]$ of $f(e,\cdot)$ at $x$ satisfies $s\leq 1=L(e)$ (recall Proposition 4.3 in \cite{GoodwinLewisLopezAcedoNicolae2025}).

Whereas the function $F$ was approximated in \cite{BennerBacakBourguignon2014} by taking empirical averages which were subsequently minimized, the minimization problem for $F$ was approached in \cite{Bacak2018} directly using the stochastic proximal point method. By the results of the present paper, we can now alternatively apply the stochastic Busemann subgradient method \eqref{SB} directly, and hence approximate a solution to the above problem iteratively using the procedure 
\[
x_{n+1}:=r_{x_n,\xi_{n}}(s_nt_n)
\]
given a starting point $x_0\in \mathcal{T}_l$, parameters $(t_n)$ 
and $(\zeta_{n+1})$ satisfying \eqref{parameters}, and where $[\xi_n,s_n]=\mathsf{Busemann}_f(\zeta_{n+1},x_{n})$ for a suitable oracle function $\mathsf{Busemann}_f$ satisfying \eqref{A3} which can be chosen using Proposition \ref{oracleExistenceProper}.

In particular, this method assumes that we can draw i.i.d.\ samples $(\zeta_{n+1})$ of the distribution $\mu_D$, which can be done efficiently in various situations, at least for not too large $l$, using suitable priors (see e.g.\ \cite{BennerBacakBourguignon2014} and see also \cite{SainudiinYork2009} for efficient i.i.d.\ sampling over small tree spaces for a slightly different posterior distribution).

Various other applications to common statistical models over nonlinear spaces are possible, in particular to other mean problems for suitable probability measures, that is for other minimization problems for functions of the form
\[
F(x):=\int_{X} d(e,x)\,d\mu(e)
\]
over $x\in X$ with a suitable space $X$ (and $E:=X$) and a given distribution $\mu$ on $X$. Such problems are e.g.\ approached over suitable Riemannian manifolds using stochastic gradient schemes, as e.g.\ studied in \cite{ArnaudonDombryPhanYang2012}, which in particular rely on the assumption of an additional lower curvature bound that could now perhaps be avoided not only in theory but also in various applications using the above stochastic Busemann subgradient method.\\

\noindent\textbf{Acknowledgments.} I want to thank Morenikeji Neri and Thomas Powell for comments on a previous draft of this paper. I also want to thank the anonymous referees for the many helpful suggestions and interesting questions which greatly improved the paper at various places. The first version of this paper was written by the author entirely alone. For preparing the revised version the author interacted with OpenAI's GPT-5.6, mainly to find relevant related literature and to correct minor mistakes and typos. All new arguments introduced in this revised version have been first obtained by the author alone. The author takes full accountability for the work. No parts of this paper were written or worded by an LLM.


\end{document}